\documentclass[a4paper,10pt]{article}
 
\usepackage{amsmath,amssymb,amsthm}
\usepackage{xspace}
\usepackage[english]{babel}
\usepackage{hyperref}
\usepackage{color}
\usepackage[T1]{fontenc}
\usepackage[utf8]{inputenc}
\usepackage[margin=1in]{geometry}
\usepackage{bm}
\usepackage{float}
\floatstyle{boxed}
\newfloat{pseudocode}{thb}{pseudo}
\usepackage{enumitem}

\def\showcomments{True}
\usepackage{amsmath,amsfonts,amsthm}
\usepackage{mathtools}

\usepackage{imakeidx}
\usepackage{multirow}
\usepackage{booktabs}

\usepackage{algorithm}
\usepackage[noend]{algpseudocode}

\usepackage{thmtools,thm-restate}

\usepackage{todonotes}

\newtheorem{definition}{Definition}[section]
\newtheorem{lemma}[definition]{Lemma}

\newtheorem{fact}[definition]{Fact}
\newtheorem{theorem}[definition]{Theorem}

\newtheorem*{theorem*}{Theorem}
\newtheorem*{lemma*}{Lemma}

\renewcommand{\le}{\leqslant}
\renewcommand{\leq}{\leqslant}
\renewcommand{\ge}{\geqslant}
\renewcommand{\geq}{\geqslant}

\ifx\showcomments\undefined

\else

\fi

\newcommand{\bigO}{\mathcal{O}}
\newcommand{\polylog}{\mathrm{polylog}}

\newcommand{\Prc}[1]{\mathbf{Pr} \left( #1 \right)}

\newcommand{\Expcc}[1]{\mathbf{E} \left[ #1 \right]}

\newcommand{\N}{\mathcal{N}}
\newcommand{\diam}{\text{diam}}

\newcommand{\enqueue}{\texttt{enqueue}}
\newcommand{\dequeue}{\texttt{dequeue}}

\algnewcommand\algorithmiccase{\textbf{case}}
\algdef{SE}[CASE]{Case}{EndCase}[1]{\algorithmiccase\ #1}{\algorithmicend\ \algorithmiccase}%
\algtext*{EndCase}%

\newcommand{\Hl}[1]{{#1}^ {(\ell)}}

\newcommand{\Vl}{V^{(\ell)}}

\newcommand{\lgraph}{$\ell$-graph}
\newcommand{\Mgra}[1]{{#1}^{(\ell)}}

\newcommand{\Erdos}{Erd\H{o}s-R\'enyi}

\newcommand{\IC}{$\mathrm{IC}$\xspace}

\newcommand{\SIR}{$\mathrm{SIR}$\xspace}

\newcommand{\Gnq}{$\mathcal{G}_{n,q}$\xspace}

\newcommand{\SWGpl}{$\mathcal{SW}(n, \alpha)$\xspace}

\title{Bond Percolation in Small-World Graphs with Power-Law Distribution}

\date{}

\author{Luca Becchetti\\
    {\footnotesize{}Sapienza Università di Roma}\\
    {\footnotesize{} Rome, Italy}\\
    {\footnotesize{}\texttt{becchetti@dis.uniroma1.it}} 
    \and Andrea Clementi \\ 
    {\footnotesize{}Università di Roma Tor Vergata}\\
    {\footnotesize{} Rome, Italy}\\
    {\footnotesize{}\texttt{clementi@mat.uniroma2.it}} 
    \and Francesco Pasquale \\
    {\footnotesize{}Università di Roma Tor Vergata}\\
    {\footnotesize{} Rome, Italy}\\
    {\footnotesize{}\texttt{pasquale@mat.uniroma2.it}} 
    \and Luca Trevisan \\
    {\footnotesize{}Università Bocconi}\\
    {\footnotesize{} Milan, Italy}\\
    {\footnotesize{}\texttt{l.trevisan@unibocconi.it}} 
    \and Isabella Ziccardi\\
    {\footnotesize{}Università dell'Aquila}\\
    {\footnotesize{}L'Aquila, Italy}\\
    {\footnotesize{}\texttt{isabella.ziccardi@graduate.univaq.it}} 
}
\begin{document}

\maketitle

\begin{abstract}
\emph{Full-bond percolation} with parameter $p$ is the process in which, given
a graph, for every edge independently, we keep the edge with probability $p$
and delete it with probability $1-p$.  Bond percolation is motivated by
problems in mathematical physics and it is studied in parallel computing and
network science to understand the resilience of distributed systems to random
link failure and the spread of information in networks through unreliable
links.

Full-bond percolation is also equivalent to the \emph{Reed-Frost process}, a
network version of \emph{SIR} epidemic spreading, in which the graph represents
contacts among people and $p$ corresponds to the probability that  a contact
between an infected person and a susceptible one causes a transmission of the
infection.

We consider \emph{one-dimensional power-law small-world graphs} with parameter
$\alpha$ obtained as the union of a cycle with additional long-range random
edges: each pair of nodes $(u,v)$ at distance $L$ on the cycle is connected by
a long-range edge $(u,v)$, with probability proportional to $1/L^\alpha$. Our
analysis determines three phases for the percolation subgraph $G_p$ of the
small-world graph, depending on  the value of $\alpha$.

\begin{itemize}
\item If $\alpha < 1$, there is a $p<1$ such that, with high probability, there
are $\Omega(n)$ nodes that are reachable in $G_p$ from one another in
$\bigO(\log n)$ hops;
\item If $1 < \alpha < 2$, there is a $p<1$ such that, with high probability,
there are $\Omega(n)$ nodes that are reachable in $G_p$ from one another in
$\log^{\bigO(1)}(n)$ hops;
\item If $\alpha > 2$, for every $p<1$, with high probability all
connected components of $G_p$ have size $\bigO(\log n)$.
\end{itemize}

Percolation in power-law small-world  graphs is well studied in the case of
infinite graphs. In finite graphs, previous work has addressed the long-range
percolation process, in which the ring edges are all preserved with probability
one. The setting of full-bond percolation in finite graphs studied in this
paper, which is the one that corresponds to the network SIR model of epidemic
spreading, had not been analyzed before.
\end{abstract}

\newpage
\section{Introduction} \label{sec:newintro}
Given a graph $G=(V,E)$ and a probability $p(e)$ associated to each edge $e\in
E$, the \emph{full-bond percolation} process\footnote{We simply write bond
percolation when no confusion arises.} on $G$ is the construction of a random
subgraph $G_p = (V,E_p)$ of $G$, called the percolation graph, obtained by
selecting each edge $e \in E$ to belong to $E_p$ with probability $p(e)$,
independently of the other edges. Often, the focus is on the {\em homogeneous}
case in which all probabilities $p(e)$ are equal to the same parameter $p$.
The main questions of interest in this case are, depending on the choice of $G$
and $p$, what is the typical size of the connected components of $G_p$ and the
typical distances between reachable vertices.

The study of the percolation process originates in mathematical
physics~\cite{KetAl80, newman1986one, VK71} and it has several applications in
parallel and distributed computing  and network science~\cite{ABS04, CPGE19,
KKT15, EK10, KNT94, LRSV14}. For example, the study of network reliability in
the presence of random link failures is equivalent to the study of the
connectivity properties of the percolation graph of the network of
links~\cite{KNT94, KL19}. 

The {\em independent cascade} is a process that  models the spread of
information and the influence of individual choices on others in social
networks, and it is equivalent to a percolation process in a way that we
explain next. In the independent cascade, we have a network $G = (V,E)$, an
{\em influence} probability $p(e)$ associated to each edge, and a set $I_0
\subseteq V$ of network nodes that initially have a certain opinion\footnote{Or
hold a certain piece of information, or perform a certain action, these are all
equivalent views that lead to the same distributed process.}. The process
evolves over time according to the following natural local rule: if a node $u$
acquired the opinion at time $t$, the node $v$ does not have the opinion, and
the edge $(u,v)$ exists in $G$, then node $u$ will attempt to convince node $v$
of the opinion, and it will succeed with probability $p(u,v)$. All  nodes that
were successfully convinced by at least one of their neighbors at time $t$ will
acquire the opinion, and will attempt to convince their neighbors at time $t+1$
and so on. The independent cascade  is studied in~\cite{KKT15}, where the
problem of interest is to find the most ``influential'' initial set $I_0$. The
resulting epidemic process is shown in~\cite{KKT15} to be equivalent to
percolation in the following sense: the distribution of nodes influenced by
$I_0$ in the independent cascade process has the same distribution as the set
of nodes reachable from $I_0$ in the percolation graph of $G$ derived using the
probabilities $p(e)$ (this important equivalence and its useful consequences
are presented in Appendix~\ref{sec:equi-bpic}).

The \emph{Reed-Frost} process is one of the simplest and cleanest models of
\emph{Susceptible-Infectious-Recovered} (\emph{SIR}) epidemic spreading on
networks~\cite{VespietAl15, Wang_2017}.  This process is identical to
independent cascade with a fixed probability $p(e) = p$ for all edges $e$. We
can interpret nodes that acquired the opinion in the previous step as {\em
Infectious} nodes that can spread the disease/opinion, nodes that do not have
the disease/opinion as nodes {\em Susceptible} to the infection, and nodes that
acquired the disease/opinion two or more steps in the past as {\em Recovered}
nodes that do not spread the disease any more and are immune to it. The
probability $p$ corresponds to the probability that a contact between an
infectious person and a susceptible one causes a transmission of the disease
from the former to the latter. The set $I_0$ is the initial set of infectious
people at time $0$. This process, being equivalent to independent cascade, is
also equivalent to percolation~\cite{KKT15}: the distribution of nodes that are
infected and eventually recover in the Reed-Frost process is the same as the
distribution of the set of nodes reachable from $I_0$ in the percolation graph
of $G$ derived using the probabilities $p(e)$. Furthermore, the distribution of
nodes that are infectious at time $t$ is precisely the distribution of nodes at
distance $t$ from $I_0$ in the percolation graph (see
Appendix~\ref{sec:equi-bpic} and~\cite{CWLC13}).

We are interested in studying full-bond percolation (and hence reliability
under random link failure, independent cascade, and Reed-Frost epidemic
spreading) in one-dimensional small-world graphs with power-law distribution of
edges. We shall now explain what each of these terms mean. {\em Small-world
graphs} are a collection of generative models of graphs, designed to capture
certain properties of real-life social networks~\cite{EK10, VespietAl15} and
communication networks~\cite{GGT03}. In the model introduced by Watts and
Strogatz~\cite{watts1998collective}, the network is obtained as a
one-dimensional or two-dimensional lattice in which certain edges are re-routed
to random destinations. In a refined model introduced by
Kleinberg~\cite{CDHKNS01, EK10}, a possible edge between two nodes at distance
$L$ in the lattice exists with probability proportional to $L^{-\alpha}$, where
the exponent $\alpha$ is a parameter of the model. 

We study full bond percolation in a variant of Kleinberg's model that has
already been adopted in several previous papers~\cite{BenjaminiB01, Bis04,
CDHKNS01, EK10, K00, VespietAl15} to study bond percolation and epidemic
processes: a \emph{one-dimensional power-law small-world  graph} with exponent
$\alpha$ is the union of a cycle with a collection of random edges, that we
call {\em bridges}. For each pair of vertices $(u,v)$ that are at distance $L
\geq 2$ on the cycle, the edge $(u,v)$ is chosen to be part of the graph with
probability proportional to $L^{-\alpha}$. The choices for different pairs are
mutually independent\footnote{This is the main difference with respect to
Kleinberg's model, in which the edges are not mutually independent.} and we
choose the factor of proportionality so that on average each vertex has one of
the random edges incident on it (so that the overall average degree of the
graph is 3, counting the cycle edges). This variant of small-world graphs has
been adopted, for instance,  to discover  structural  key-properties that
determine the performances of some important diffusion and navigation problems
in real networks~\cite{WSZH15}. 

Before proceeding with a statement of our results, we formally define our
generative model for future reference.

\begin{definition}[1-D power-law small-world  graphs]\label{def:small-world}
For every $n \geq 3$ and $\alpha>0$, a random graph $G=(V,E)$ with
$V=\{0,\dots,n-1\}$ is sampled according to the distribution \SWGpl\ if  $E =
E_1 \cup E_2$, where: $(V,E_1)$ is a cycle and its edges are called
\emph{ring-edges}, and, for each pair of non-adjacent vertices $u, v\in V$, the
\emph{bridge} $(u, v)$ is included in $E_2$ independently, with probability
\[
\Prc{(u,v) \in E_2 }= \frac{1}{d(u,v)^{\alpha}}\cdot \frac{1}{C(\alpha,n)},
\]
where $d(\cdot,\cdot)$ is the shortest-path distance in the ring and
$C(\alpha,n)$ is the normalizing constant\footnote{Note that $C(n, \alpha)$ is
not, strictly speaking, a constant, but rather a normalizing factor that
depends on both $\alpha$ and $n$. It is always upper bounded by an absolute
constant across the entire range of $\alpha$, while it falls within an interval
bounded by two constants when $\alpha > 1$. For the sake of conciseness,
abusing terminology we write ``constant'' instead of normalizing factor.}
\[
C(\alpha,n)=2\sum_{x=2}^{n/2}\frac{1}{x^{\alpha}} \, .
\]
\end{definition}

The process of {\em long-range percolation} on finite power-law small-world
graphs is the variant of the percolation process applied to the generative
model described above in which ring edges are preserved with probability one.
The resulting percolation graphs are always connected, and the main question of
interest is their diameter. Long-range percolation in power-law  small-world
graph distributions is well understood, and the one-dimensional case is studied
in~\cite{BenjaminiB01, K00, watts1998collective}. In
particular,~\cite{BenjaminiB01} provide bounds on the diameter and on the
expansion of such graphs as a function of the power-law exponent $\alpha$. Such
results have been then sharpened and generalized to multi-dimensional boxes
in~\cite{Bis11}. The long-range percolation process, however, is not a
realistic generative model for epidemiological processes, because even the most
contagious diseases, including Ebola, do not have $100\%$ probability of
spreading through close contacts (see~\cite{moore2000epidemics,
newman1999scaling, VespietAl15} for a discussion of this point).

Full-bond percolation in power-law small-world graphs has been studied in the
case of infinite lattices, including the one-dimensional case that is the
infinite analog of the model that we study in this paper. In the infinite case,
the main questions of interest, which are studied in~\cite{Bis04}, are whether
the percolation graph has an infinite connected component and, given two
vertices, what is their typical distance in the percolation conditioned on them
both being in the infinite component, as a function of their distance in the
lattice.  Although there are similarities, techniques developed to study
infinite percolation graphs do not immediately apply to the finite case,
particularly when one of the technical  goals is to obtain rigorous results in
concentration.

\smallskip In this paper, we study full-bond percolation in power-law
small-world graphs, a key question left open by previous research. The rest of
this paper is organized as follows. Section~\ref{se:contribution} provides an
overview of our results and techniques we use. A more extensive review of
previous work is given in Section~\ref{sec:related_new}.
Section~\ref{se:prelim} introduces notation and some key preliminary notions,
while Sections~\ref{sec:alpha>2}, \ref{sec:1<alpha<2} and~\ref{sec:alpha<1}
provide the full analysis of full-bond percolation in one-dimensional
power-law, small-world graphs. Finally, the proofs of some technical results
are presented in the appendix.

\section{Our Contribution}\label{se:contribution}
Previous work on full-bond percolation in power-law small-world infinite graphs
and on long-range percolation in  power-law small-world finite graphs
identified three phases for the value of $\alpha$. Infinite graphs exhibit an
infinite connected component after full-bond percolation when $\alpha < 2$ and
no infinite connected component when $\alpha >2$; furthermore distances within
the infinite component scale differently in the $\alpha <1$ versus the $1 <
\alpha < 2$ case. Finite graphs remain connected under long-range percolation
(because the ring-edges are not subject to percolation), but distances scale
differently in the $\alpha <1$, in the $1 < \alpha < 2$ and the $\alpha > 2$
case. We postpone a fuller discussion of previous work in these models to
Section~\ref{sec:related_new}.

Our analysis of the full-bond percolation process on power-law small-world
finite graphs shows that the process exhibits different behaviours in the same
three phases. The behaviors are of a somewhat different kind than in previous
work. In the $\alpha <2$ case, we find, for $p$ sufficiently large, a connected
component of size $\Omega_{\alpha} (n)$, which is best possible, and in the
$\alpha > 2$ case we show that all connected components have size
$\bigO_{\alpha}(\log n)$. Furthermore, when $\alpha <1$, we identify
$\Omega_\alpha(n)$ vertices all reachable from each other in $\bigO_\alpha(\log
n)$ steps, and when $1 < \alpha < 2$ we identify $\Omega_\alpha(n)$ vertices
all reachable from each other in $\bigO(\log^{\bigO_\alpha(1)}(n))$ steps.

Our analysis is fully rigorous and the obtained results hold \emph{with high
probability}\footnote{As usual, we say that an event $E$ holds with high
probability if a constant $\gamma>0$ exists such that $\Prc{E}\geq
1-n^{-\gamma}$.} (for short, \emph{w.h.p.}).

The three phases are characterized by sharply different distributions of the
typical length (measured according to the ring metric) of the bridges. To gauge
the difference, consider the expectation, for a fixed vertex, of the sum of the
lengths (in ring metric) of the bridges incident on the vertex, and call this
expectation $BL_{\alpha,n}$.

When $\alpha< 1$, we have that $BL_{\alpha,n}$ is linear in $n$. When $1 <
\alpha < 2$, then $BL_{\alpha,n}$ is of the form $O(n^{2-\alpha})$, going to
infinity with $n$, but sublinearly in $n$. Finally, when $\alpha >2$,
$BL_{\alpha,n}$ is a constant  that depends only on $\alpha$ and is independent
of $n$. Nodes have, in expectation, only one bridge, so $BL_{\alpha,n}$ is an
indication of how much we can advance on the ring by following one bridge. 

When $\alpha <1$, the bridges  are basically as good as random edges, and we
would expect a giant component to emerge even after full-bond percolation, if
$p$ is large enough. When $\alpha >2$, the bridges   behave like a constant
number of ring-edges, and we would not expect a large component when $p<1$. The
case $1< \alpha < 2$ is the one for which it is hardest to build intuition, and
the fact that the bridges   have typically length of the form $n^{1-
\Omega(1)}$ might suggest that it would take $n^{\Omega(1)}$ steps to reach
antipodal nodes.

Previous work on long-range percolation had established a polylogarithmic
diameter bound in the model in which ring-edges are not subject to percolation.
In that model,  all pairs of nodes are reachable in a polylogarithmic number of
steps even though the typical bridge  has length $n^{1-\Omega(1)}$, and this
suggests that the shortest path structure is such that a small number of long
bridges  is used by several shortest paths. One thus  would expect such a
structure to be sensitive to full-bond percolation, and indeed the proof
of~\cite{BenjaminiB01} relies on the deterministic presence of the ring-edges.

Instead, we prove that, when $1 < \alpha <2$, w.h.p, most pairs of nodes are
reachable from one another in a polylogarithmic number of hops after the
full-bond percolation process.

\subsection*{Our results}\label{sec:ours}
We next provide the analytical results for each of the three regimes.

When $\alpha>2$, we  show that the percolation graph $G_p$, for every possible
percolation probability $p<1$,  has w.h.p. all connected components of size at
most $\bigO(\log n)$ and of small ring-diameter.

\begin{theorem}[Case $\alpha \in (2, +\infty)$]
\label{thm:alfa>2_terminates}

Let $\alpha>2$ be a constant and $p<1$ a percolation probability. Sample a
graph $G=(V,E)$ from the \SWGpl distribution and let $G_p$ be the percolation
graph of $G$ with percolation probability $p$. The following holds:

\begin{enumerate}
\item W.h.p., the connected components of $G_p$ have size at most $\bigO(\log
n)$;
\item For each node $v \in V$ and  for any sufficiently large $\ell$, with
probability  $1-\Omega( \ell^{-(\alpha-2)/2})$, every node connected to $v$ in
$G_p$ Is at ring-distance no larger than $\bigO(\ell^2)$ from $v$.
\end{enumerate}
\end{theorem}

From an epidemiological point of view, this first regime is thus characterized
by a negligible chance to observe an  outbreak according to the Reed-Frost
process,  even in the presence of a large  number (say some small  root of $n$)
of initial infected nodes (i.e. sources). In particular, the second claim of
the theorem above strongly bounds the possible infected area of the ring graph. 

The following case, determined by the range $1<\alpha<2$, shows the most
interesting behaviour. 
 
\begin{theorem}[Case $\alpha \in (1,2)$]
\label{thm:main_1alpha2}
Let $\alpha \in (1,2)$ be a constant and $p$ a percolation probability. Sample
a graph $G=(V,E)$ from the \SWGpl distribution and let $G_p$ be the percolation
graph of $G$ with percolation probability $p$. Then, constants
$\underline{p},\overline{p} \in (0,1)$ (with $\underline{p} \leq \overline{p}$)
exist such that the following holds:

\begin{enumerate}
\item If $p> \overline{p}$, w.h.p., there exists a set of $\Omega(n)$ nodes
that induces a connected sub-component in $G_p$ with diameter
$\bigO(\polylog(n))$; 
\item If $p< \underline{p}$, w.h.p. all the connected components of $G_p$ have
size $\bigO(\log n)$.
\end{enumerate}
\end{theorem}
The first claim above  implies that, if $p$ is sufficiently large (but still a
constant smaller than $1$), then, there is a good  chance that few source nodes
are able to infect a large (i.e. $\Omega(n)$) number of nodes and, importantly,
this outbreak takes place at an almost exponential speed.

Finally, when $\alpha<1$, we show the emergence of a behaviour similar to that
generated by one-dimensional small-world models with bridges selected according
to the \Erdos\ distribution \cite{becchettiCDPTZ21, MCN00}.

\begin{theorem}[Case $\alpha \in (0,1)$]
\label{thm:main_1alpha}
Let $\alpha \in (0,1)$ and $p$ a percolation probability.  Sample a graph
$G=(V,E)$ from the \SWGpl distribution and let $G_p$ be the percolation graph
of $G$ with percolation probability $p$. Then, constants
$\overline{p},\underline{p} \in (0,1)$ (with $\underline{p} \leq \overline{p}$)
exist, such that the following holds: 

\begin{enumerate}
\item If $p> \overline{p}, $ w.h.p., there exists a set of $\Omega(n)$ nodes
that induces a connected sub-component in $G_p$ with diameter $\bigO(\log n)$.
Moreover, for a sufficiently large $\beta >0$, for any   subset $S \subseteq V$
of size $|S| \geqslant \beta \log n$, the subset of nodes  within distance
$\bigO(\log n)$ in $G_p$ from   $S$ has size $\Omega(n)$, w.h.p.
\item If $p< \underline{p}$, w.h.p. all the connected components of $G_p$ have
size $\bigO(\log n)$. 
\end{enumerate}
\end{theorem}

In this last case, the presence of a sparse subset of relatively-long random
bridges implies, above the probability threshold, that a few (i.e. $\Omega(\log
n)$) sources w.h.p. generate a large outbreak at exponential speed.

\subsection*{Technical contribution: an overview}\label{ssec:ourtech}
The analysis techniques we use, like our main results, differ as the parameter
$\alpha$ varies. We recall that $G_p$ is the percolation graph of $G$, sampled
from the \SWGpl distribution.

When $\alpha>2$, we show that each component of the percolation graph $G_p$ has
w.h.p. at most $\bigO(\log n)$ nodes. To prove this fact, we need to cope with
the complex ``connectivity'' of $G_p$ yielded by the percolation of both
ring-edges and the random bridges. To better  analyze this  structure, we
introduce the notion of \emph{$\ell$-graph} $G_p^{(\ell)}$ of $G_p$, where
$\ell$ is any fixed  integer $\ell>0$. This new graph $G_p^{(\ell)}$ is in turn
a one-dimensional small-world graph of $n/\ell$ ``supernodes''. It is defined
by any fixed partition of $V$ of disjoint ring intervals, each of them formed
by  $\ell$ nodes: such $n/\ell$ intervals are the nodes, called
\emph{super-nodes}, of  $G_p^{(\ell)}$. The set of  edges of $G_p^{(\ell)}$ is
formed by two types of random links: the \emph{super-edges} that connect two
adjacent super-nodes and the \emph{super-bridges} connecting two non adjacent
super-nodes of the $n/\ell$-size ring (for the formal definition of
$G_p^{(\ell)}$ see Definition \ref{def:l-graph}). We then prove that, for any
$\alpha >2$ and any $p<1$,  each super-node of this graph has:

\begin{enumerate}
\item Constant probability to have no neighbors (for every value of $\ell $);
\item Probability $\bigO(1/\ell^{\alpha-2})$ to be incident to a super-bridge.
\end{enumerate}

We then design an appropriate BFS visit of $G_p^{(\ell)}$ (see
Algorithm~\ref{alg:BFS-visit} in Section~\ref{sec:alpha>2}) that keeps the role
of super-edges and super-bridges well-separated. In more detail, starting from
a single super-node $\mbox{\textsc{s}}$, we show that the number of super-nodes
explored at each iteration of the visit  turns out to be dominated by a
branching process having two distinct additive contributions: one generated by
the new, visited percolated super-edges and the other generated by the new,
visited percolated super-bridges. Special care is required to avoid too-rough
redundancy in counting possible overlapping contributions from such two
experiments. Then, thanks to Claims~1 and~2 above,  we can prove that, for
$\ell=\bigO(1)$ sufficiently large, w.h.p. this branching process ends after
$\bigO(\log n)$ steps, and this proves that the connected component of
$\mbox{\textsc{s}}$ in $G_p^{(\ell)}$ has  $\bigO(\log n)$ super-nodes.   We
also remark that our approach above   also implies the following interesting
result: with probability $1-\bigO(1/\ell^{\alpha-2})$, all super-nodes
connected to $v$ are at ring distance at most $\ell$ from $v$. This implies
that, for each node $s$, the nodes in the connected component of $s$ in $G_p$
are within ring distance $\bigO(\ell^2)$ from $s$, with the same probability.

When $1<\alpha<2$, we make a different analysis, based on a suitable  inductive
interval partitioning of the initial ring. In this case, our technique is
inspired by the partitioning method used in~\cite{newman1986one}
and~\cite{BenjaminiB01}. However, though the high level idea is similar, our
analysis needs to address technical challenges that require major adjustments.
In particular, the analysis of~\cite{newman1986one} only addresses and relies
upon properties of infinite one-dimensional lattices with additional long-range
links, applying Kolmogorov 0-1 law to the tail event of percolation. On the
other hand,~\cite{BenjaminiB01} borrows from and partly
extends~\cite{newman1986one} to the finite case, proving a  result similar to
ours for finite one-dimensional lattices, but with important differences both
in the underlying percolation model and in the nature of the results they
obtain.  As for the first point, in their model, ring-edges are
deterministically present, while only long-range edges (what we call bridges in
this paper) are affected by percolation. In contrast, in our model, \emph{all}
edges are percolated to obtain a realization of $G_p$. Consequently, the graph
is deterministically connected in~\cite{BenjaminiB01}, while connectivity  is
something that only occurs with some probability in our setting. In particular,
we face a subtler challenge, since we need to show the existence of a connected
component that both spans a constant fraction of the node \emph{and} has
polylogarithmic diameter\footnote{As an example, we might have a connected
component including a constant fraction of the nodes and linear diameter,
containing a smaller connected component, still spanning a constant fraction of
the nodes, but of polylogarithmic diameter. }. Moreover, their result for the
diameter of the percolated graph only holds with probability tending to $1$ in
the limit, as the number of nodes grows to infinity, while we are able to show
high probability. 

We next provide a technical overview of our approach, highlighting the 
main points of the proof in which our percolation model requires a 
major departure from the approach of~\cite{BenjaminiB01}.  

Similarly to~\cite{BenjaminiB01}, we consider a diverging sequence $\{N_k\}_k$
and for each $k$ we consider in $G_p$ an interval of length $N_k$ consisting of
adjacent nodes on the ring. Departing from~\cite{BenjaminiB01}, we then prove
that each interval of size $N_k$ contains a constant fraction
$(1-\varepsilon_k)N_k$ of nodes,  that induces a connected component of
diameter $D_k$, with probability $1-\delta_k$, considering only edges
\emph{internal} to the interval and were $\varepsilon_k$ and $\delta_k$ are
suitable constants that only depend on $k$, $p$ and $\alpha$.  To prove this
fact, for each $k$, we consider an arbitrary interval $I_k$ of size $N_k$ and
we proceed inductively as follows:

\begin{enumerate}
\item We divide $I_k$ in $N_k / N_{k-1}$ smaller intervals of size $N_{k-1}$;
\item We assume inductively that each of these intervals contains
$(1-\varepsilon_{k-1})N_{k-1}$ nodes that induces a connected component of
diameter $D_{k-1}$, with probability $1-\delta_{k-1}$;
\item With a concentration arguments, we prove that a constant fraction of the
intervals in which we divide $I_k$ have the above property, with probability
that increases with $k$ (we call these intervals \emph{good});
\item We consider all good intervals of the previous point, and we prove that
all these intervals are connected to each other with probability increasing in
$k$.
\end{enumerate}

The above reasoning therefore implies, for some $\varepsilon_k$, the existence
of a fraction $(1-\varepsilon_k)N_k$ of nodes in $I_k$, that with increasing
probability in $k$, induce a connected component of diameter $D_k = 2D_k +1$.
More details about the sequences $N_k,\varepsilon_k$ and $\delta_k$ can be
found in the full proof.  As mentioned earlier, our proof needs to specifically
address percolation of both ring-edges and bridges. This in particular, means
tackling points (2) - (4) above, which is a major challenge not present
in~\cite{BenjaminiB01} and is taken care of in the technical
Lemma~\ref{lem:key_1alpha2} and in the proof of the main
Lemma~\ref{lem:1alpha2_underthreshold} itself.  In particular, considering only
edges internal to the interval is very important, since, in this way, the
events denoting the connection of disjoint intervals are independent.  The
recurrence of the diameter derives from the fact that the path connecting two
nodes $u$ and $v$ in $I_k$ consists of the following three sub-paths:

\begin{enumerate}
\item The first part has length at most $D_k$, and it is a path in the
aforementioned ``good'' interval of size $(1-\varepsilon_{k-1})N_{k-1}$
containing $u$;
\item The second part consist of a single edge, connecting the good interval
containing $u$ with the good interval of $v$;
\item The third has length at most $D_k$, and is a path of the good interval of
size $(1-\varepsilon_{k-1})N_{k-1}$ containing $v$.
\end{enumerate}

Finally (Lemma \ref{lem:1alpha2_underthreshold}), we consider $k$ such that
$N_k=\Omega(n)$. For such $k$, we prove that
\[
\delta_k 
= \bigO(1/n^\varepsilon), \quad \varepsilon_k  \ll 1, \quad D_k 
= \polylog(n),
\]
so this implies the existence in $G_p$ of $\Omega(n)$ nodes inducing a
connected component with diameter $\polylog(n)$, w.h.p.

When $\alpha<1$, we notice that every bridge-edges $(u,v)$ is in $G_p$ with
probability at least $pc/n$, for some constant $c$. This reduces our problem to
the analysis of the percolation graph $H_p$ of a graph $H$, where $H$ is the
union of a ring and an \Erdos \ graph $\mathcal{G}_{n,q}$.
\cite{becchettiCDPTZ21} contains a detailed analysis of this case, and we use
their results to prove that $G_p$, for any sufficiently large $p$, contains a
connected component with diameter $\bigO(\log n)$.

We finally remark that, for every value $\alpha$, when $p<1/3$ the connected
components of $G_p$ have at most $\bigO(\log n)$ nodes. This fact easy follows
from concentration techniques, and the fact that every node in $G_p$ has
expected degree at most $3$.

\section{Related work}\label{sec:related_new}
As discussed in the previous sections, variants of bond-percolation have been
investigated in different areas, including computer and network science,
mathematical physics, and epidemiology. In the rest of this section, we only
discuss rigorous analytical results that are most related to our work.

\subsection*{Small-world graphs}
\paragraph{Power-law graphs.} As we briefly discussed in the introduction, a
number of variants were considered in previous work, differing in key aspects.
E.g., whether the graph has infinite size (e.g., $\mathbb Z^d$) or bounded
size, the dimension $d$ of the underlying lattice/torus  and, very importantly,
whether bond-percolation is applied to all edges (\emph{full-bond percolation})
or only to a subset of long-range edges (\emph{long-range bond percolation}). 
 
Long-range percolation in finite small-world graphs with power-law bridge
distributions was investigated in~\cite{BenjaminiB01, K00, watts1998collective}
in the one-dimensional case. In particular,~\cite{BenjaminiB01} provides bounds
on the diameter and on the expansion of the augmented ring as a function of the
power-law exponent $\alpha$. These results have been then sharpened and
generalized to multi-dimensional boxes by Biskup in~\cite{Bis11}. These results
are consistent  with the three regimes we characterize in this paper for the
full-bond percolation process. However, none of our results can be derived from
these previous results because of a crucial aspect we discussed in the previous
section: in long-range percolation models, the percolation graph is
deterministically connected and arguments of the works above, showing upper
bounds on the diameter, strongly rely on the existence of sub-paths including
that include deterministic grid edges. Importantly, such previous analyses
leave open the key question of whether or not, in the range $\alpha < 2$,
full-percolation can still result (for suitable values of the percolation
probability $p$) in the emergence of a large connected subgraph of small
diameter: an affirmative answer to this question and its consequences are in
fact one of the main novelties of our work. On the other hand, lower bounds on
the diameter of the connected component given in~\cite{BenjaminiB01, Bis11} do
not exclude the existence of a large  connected component of smaller diameter
and, thus, a potential epidemic outbreaks. Finally, the bounds given
in~\cite{BenjaminiB01, Bis11}, are not proved in ``concentration'' (i.e.,
w.h.p.).
    
As for full-bond percolation on infinite graphs,~\cite{Bis04} studies infinite
lattices of any dimension $d$ augmented with long-range bridges selected
according to a class of distributions that includes  power-laws. In more
detail, they derive bounds on the ``typical distance'' between any two points
of the grid that are consistent with our results for finite graphs.  However,
as discussed in the previous section, their argument to prove both upper and
lower bounds on the typical distance (and, so, on the diameter of any induced
subgraph) does not extend to finite graphs. Moreover, no explicit bounds on the
confidence probability of their bounds can be derived as a functions of the
size of the induced subgraphs, which is a major technical goal in the finite
setting.
    
\paragraph{Other models of small-world graphs.} \cite{becchettiCDPTZ21} studies
a small-world model resulting from a  cycle augmented with a subset of bridge
edges sampled from the \Erdos\ distribution \Gnq. They provide a sharp
threshold for the percolation probability $p$ as a function of the
bridge-probability parameter $q$. As remarked in Subsection \ref{ssec:ourtech},
our analysis for the case $\alpha < 1$ relies on some technical results
from~\cite{becchettiCDPTZ21}, since in this case, the distribution of bridges
is very close to \Gnq, for a suitable choice of the parameter $q$.
 
A further analysis of the percolation process in~\cite{becchettiCDPTZ21}
considers the case in which the graph consists of a cycle and a random perfect
matching. Also for this regular case, the authors prove a sharp threshold for
the percolation probability $p$: it turns out that this threshold is smaller
than the one for the  \Erdos\ small-world graphs when $q$ is chosen, so as to
yield the same expected number of bridges.
    
Bond percolation on the class of $1$-dimensional small-world networks (that is,
graphs obtained as the union of a cycle and of randomly chosen edges) has been
studied in~\cite{MCN00}: using numerical approximations on the moment
generating function, non-rigorous bounds on the critical threshold have been
derived while analytical results are given neither for the expected size of the
number of infected nodes above the transition phase of the independent-cascade
process nor for the infection time.  Further non-rigorous results on the
critical points of several classes of complex networks have been derived
in~\cite{LRSV14, LKAK16} (for good surveys see~\cite{VespietAl15, Wang_2017}).

\subsection*{Other classes of graphs}
In~\cite{DGM06}, for a symmetric, connected graph $G = (V,E)$, Draief et al
prove a general lower bound on the critical point  in terms of spectral
properties. Further versions of such bounds for special cases have been
subsequently derived in~\cite{LKAK16, LSV14}. We remark here that such
approaches are not useful to discover different behaviours depending on the
value $\alpha$ in our model. 

A fundamental and rigorous study of bond percolation in random graphs has been
proposed by Bollob\'as et al in~\cite{Bollo07}. They establish a coupling
between the bond percolation process and a suitably defined branching process.
In the general class of inhomogenous \Erdos\ random graphs, calling $Q = \{
q_{u,v}\}$ the edge-probability matrix, they derived the critical point
(threshold) of the phase transition and the size of the giant component above
the transition. The class of inhomogeneous random graphs to which their
analysis applies includes generative models that have been studied in the
complex network literature. For instance, a version of the Dubin's
model~\cite{Durrett1990ATT} can be expressed in this way, and so can the
\emph{mean-field scale-free model}~\cite{Bollobs2004TheDO}, which is, in turn,
related to the   Barabási–Albert model~\cite{Barab509}, having the same
individual edge probabilities, but with edges present independently. Finally,
we observe that the popular CHKNS model introduced by  Callaway et
al~\cite{CDHKNS01} can be analyzed using an edge-independent version of this
model. Indeed, they consider a random graph-formation process where, after
adding each node, a Poisson number of edges is added to the graph, again
choosing the endpoints of these edges uniformly at random. For all such
important classes of random graph models, they show tight bounds for the
critical points and the relative size of the giant component beyond the phase
transition.

In our setting, if we sample a graph from \SWGpl and then consider the
percolation graph $G_p$, the distribution of $G_p$ is that of an inhomogenous
\Erdos\ graph in which the cycle edges have probability $p$ and each bridge
$(x,y)$ has probability $\sim p \cdot 1/d(x,y)^{\alpha}$.

We observe that the results of~\cite{Bollo07} to the inhomogenous random graph
equivalent to the above percolation graph do not lead to tractable conditions
on the existence of threshold values of $p$ for which a large induced subgraph
of small diameter emerges, which is the kind of results  we are interested on. 

Finally, sharp threshold for bond percolation on bounded-degree expanders are
given in~\cite{BBLR12}.

\section{Model and Preliminaries}\label{se:prelim}

In this paper, we study bond percolation of graphs sampled from \SWGpl, as
formalized in Definition~\ref{def:small-world}. In this section, we define
notation, key notions and tools that will be used throughout the rest of this
paper. Further notation used in the proofs of specific results is introduced
wherever it is used. 

\medskip \noindent\textbf{Key notions and notation.} For the sake of
completeness, we begin with the formal definition of bond percolation.

\begin{definition}[Bond percolation]\label{def:percolation}
Given a graph $G = (V, E)$ and a real $p \in [0,1]$, its (bond)
\emph{percolation graph} $G_p$ is the random subgraph obtained from $G$ by
removing each edge $e \in E$ independently, with probability $1 - p$.
\end{definition}

Considered any graph $G = (V, E)$ and $v\in V$, we denote by $\N_G(v)$ the
neighborhood of $v$ in $G$, while $\deg_G(v) = |\N_G(v)|$ denotes the degree of
$v$ in $G$. We omit the subscript when $G$ is clear from context. 

If $G = (V, E_1\cup E_2)$ is sampled from \SWGpl\ as in
Definition~\ref{def:small-world}, we say that a bridge $(u,v)\in E_2$  has
\emph{length} $d(u,v)$ (where $d(\cdot,\cdot)$ is as in
Definition~\ref{def:small-world}).  We also say that $u$ is at
\emph{ring-distance} $d(u,v)$ from $v$ (and viceversa). Considered a bridge
$(u, v)$, we say that $(u, v)$ is on the \emph{clockwise} (respectively,
\emph{counter-clockwise}) side of $u$ if $d(u, v)$ corresponds to moving
clockwise (respectively, counter-clockwise) along the cycle $(V, E_1)$ from $u$
to $v$. 

Given a graph $G = (V, E)$ and $s\in V$, we denote by $\Gamma_G(s)$ the
connected component containing $s$ in $G$. We may omit $G$ when clear from
context, while with a slight abuse of notation, we simply write $\Gamma_p(s)$
for $\Gamma_{G_p}(s)$ when we refer to the percolation $G_p$ of some graph $G$,
which will always be understood from context. Given $G = (V, E)$ and
$S\subseteq V$, $\diam_G(S)$ is equal to the diameter of the subgraph of $G$
induced by $S$ if this is connected, otherwise $\diam_G(S) = \infty$. With a
slight abuse of notation, we write $\diam_p(S)$ when $G$ is the percolation
graph $G_p$.  Given $G = (V, E)$ sampled from \SWGpl\ and any subgraph $H = (V,
E')$ such that $E'\subseteq E$ (e.g., $G_p$), we associate a \emph{ring-metric}
to $H$, so that the \emph{ring-distance} between $u, v$ is simply $d(u, v)$
defined above on $G$. 

\medskip

\paragraph{Remarks.} In the sections that follow, unless stated otherwise,
probabilities are always taken over both the randomness in the sampling of $G$
from \SWGpl and over the randomness of the percolation.  We further remark
that our choice of the normalizing constant $C(\alpha,n)$ in
Definition~\ref{def:small-world} entails $\Expcc{\deg(v)}=3$, while the
following, preliminary fact follows from a straightforward application of
Chernoff bound (Theorem~\ref{thm:chernoff}):

\begin{fact}
\label{fact:degreeofSWGpl}
Sample a graph $G=(V,E)$ from \SWGpl. Then,
\begin{equation}
\Prc{\max_v \deg(v)\leq 4\log n + 2} \geq 1-\frac{1}{n} \, .
\end{equation}
\end{fact}

\medskip
 
\noindent\textbf{Galton-Watson branching processes.}
\label{sec:GW}
Our analysis of the percolation process in part relies on a reduction 
to the analysis of appropriately defined branching processes.

\begin{definition}[Galton-Watson Branching Process]
\label{def:GW}
\label{def:branchingprocess}
Let $W$ be a non-negative integer random variable, and let $\{ W_{t,1}\}_{t\geq
1, i\geq 1}$ be an infinite sequence of independent identically distributed
copies of $W$. The \emph{Galton-Watson branching process} generated by the
random variable $W$ is the process $\{ X_t \}_{t\geq 0}$ defined by $X_0=1$ and
by the recursion
\[
X_t = \sum_{i=1}^{X_{t-1}} W_{t,i} \, .
\]
All properties of the process $\{ X_t \}_{t\geq 0}$ (in particular, population
size and extinction probability) are captured by the equivalent process
$\{B_t\}_{t \geq 0}$, recursively defined as follows:
\[
	B_t = \left\{
		\begin{array}{ll}
			1, & t = 0;\\
			B_{t-1} + W_t - 1, & t > 0\ \text{and}\ 
			B_{t-1} > 0;\\
			0, & t > 0\ \text{and}\ B_{t-1} = 0 \, ,
		\end{array}\right.
\]
where $W_1,\ldots,W_t,\ldots$ are an infinite sequence of independent and
identically distributed copies of $W$.  
\end{definition}

In the remainder, when we refer to the Galton-Watson process generated by $W$,
we always mean the process $\{B_t\}_{t \geq 0}$.  In particular, we define
$\sigma=\min\{t>0: B_t=0\}$ (we set $\sigma=+\infty$ if no such $t$ exists).
Note that, for any $T< \sigma$, we have $B_T = \sum_{t=1}^T W_t -T$.

\section{The case $\alpha>2$}
\label{sec:alpha>2}

The aim of this section is to prove Theorem~\ref{thm:alfa>2_terminates}, which
follows from the two lemmas below.

\begin{lemma}
\label{lem:alfa>2_terminates}
Let $\alpha<2$ be a constant and $p<1$ a  percolation probability.  Sample a
graph $G=(V,E)$ from the \SWGpl distribution and let $G_p$ its   percolation
graph. For every $s \in V=\{0,\dots,n-1\}$, the connected component
$\Gamma_p(s)$ contains $\bigO(\log n)$ nodes with probability at least
$1-1/n^2$.
\end{lemma}

\begin{lemma}
\label{lem:alfa>2_terminates_boundconnectedcomponent}
Under the same hypotheses of Lemma \ref{lem:alfa>2_terminates}, for every $s \in V=\{0,\dots,n-1\}$ and for any sufficiently large $\ell$, 
\begin{equation}
\label{eq:prob_setS}
\Prc{\forall u\in\Gamma_p(s): d(s,u)\le 2\ell^2} 
\ge 1-\frac{8}{(\alpha-2)\ell^{(\alpha-2)/2}} \, .
\end{equation}
\end{lemma}

We remark that \eqref{eq:prob_setS} implies  that, for any increasing distance
function $\ell=\ell(n) = \omega(1)$, every node in $\Gamma_p(s)$  has
ring-distance from $s$ not exceeding $2\ell^2$ with probability $1-o(1)$.

The proof of Lemma~\ref{lem:alfa>2_terminates} is given in
Subsection~\ref{ssec:lem:alfa>2_terminates}, while  the proof of
Lemma~\ref{lem:alfa>2_terminates_boundconnectedcomponent} in
Subsection~\ref{ssec:lem:alfa>2_terminates_boundconnectedcomponent}. To prove
these lemmas, we rely on the notion of $\ell$-graph, defined in the following
paragraph, together with supplementary notation that will be used in the
remainder of this section.

\paragraph{$\ell$-graphs.} In what follows, we define a new graph on the vertex
set $V=\{0,\dots,n-1\}$, starting from any one-dimensional small-world graph
$G=(V,E_1\cup E_2)$, where $(V,E_1)$ is a cycle, which defines a ring-metric,
and $E_2$ is an arbitrary subset of bridges. In the remainder, we always assume
$n\ge 3$.

\begin{definition}[\lgraph]
\label{def:l-graph}
Let  $H = (V, E_H)$ be a subgraph of a one-dimensional small-world graph $G$,
where $E_H\subseteq E_1\cup E_2$. For $\ell\ge 1$, consider an arbitrary
partition of $V$ into $n/\ell$ disjoint intervals of length $\ell$,
$\{I^{(1)},\ldots , I^{(n/\ell)}\}$ with respect to the ring metric induced by
$(V, E_1)$.  The  \emph{\lgraph}\     associated to $H$ is the graph  $\Hl{H}
=(\Vl, E_H^{(\ell)})$, where $\Vl = \{I^{(1)},\ldots , I^{(n/\ell)}\}$, and  
\[
E_H^{(\ell)} = \{(I^{(h)}, I^{(k)}):\exists u\in 
I^{(h)}, v\in I^{(k)}\text{ s.t. } (u, v)\in E_H\} \, .
\]  
\end{definition}

A generic element in $\Vl$ thus corresponds to $\ell$ consecutive nodes in the
ring $(V, E_1)$, it is called \emph{super-node}, and, depending on the context,
is denoted by $\mbox{\sc v}$ or by its corresponding interval $I_{\mbox{\sc
v}}$ in the partition $\{I^{(1)},\ldots , I^{(n/\ell)}\}$ of $V$. Fixed a
partition, we denote with $E_1^{(\ell)}$ the set of links connecting two
adjacent super-nodes in $V^{(\ell)}$ (that is, two adjacent intervals in
$(V,E_1)$): notice that $(V^{(\ell)},E_1^{(\ell)})$ is a ring of $n/\ell$
super-nodes, called \emph{$\ell$-ring}. Then, given a subgraph $H = (V, E_H)$
as in the definition above, the elements in $E_H^{(\ell)}$ can be partitioned
into two subsets: the elements in $E_H^{(\ell)} \cap E_1^{(\ell)}$ are called
\emph{super-edges}, while all the remaining elements in $E_H^{(\ell)}$ are
called \emph{super-bridges}. $H^{(\ell)}$ can thus be seen as a subgraph of a
one-dimensional small-world graph $G^{(\ell)}$ with $n/\ell$ super-nodes,
formed  by a ring $(V^{(\ell)},E_1^{(\ell)})$, and an additional set of
super-bridges.  The example in Figure \ref{fig:l-graphs} summarizes the above
definitions.
\begin{figure*}[]
    \centering
    \includegraphics[width=0.5\textwidth]{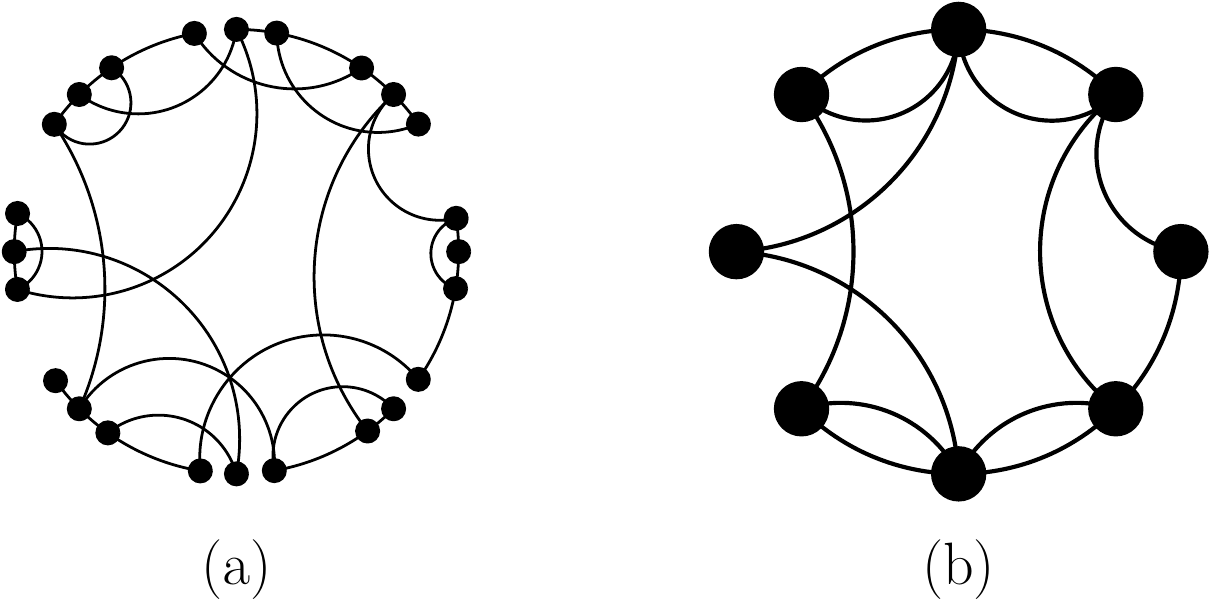}
    \caption{(a) A graph $G=(V,E)$ with a ring-metric. (b) The $3$-graph of $G$.}
    \label{fig:l-graphs}
\end{figure*}
Let $H = (V, E_H)$ as in Definition~\ref{def:l-graph} and assume $u, v\in V$,
with $u\in I^{(h)}$ and $v\in I^{(k)}$. Clearly, if $(u, v)\in E_H$ then
$(I^{(h)}, I^{(k)})\in E_H^{(\ell)}$. Moreover, the following fact
straightforwardly holds:

\begin{fact}
\label{fact:connectedcomp-GMandG}
Let $H = (V, E_H)$ as in Definition \ref{def:l-graph} and let $s\in I^{(k)}$
for some $k\in\{1,\ldots , n/\ell\}$.  Then,
\begin{equation*}
    |\Gamma_H(s)| \leq \ell	|\Gamma_{\Hl{H}}(I^{(k)})| . 
\end{equation*}
\end{fact}

The following, preliminary lemma is used in the proofs of
Lemma~\ref{lem:alfa>2_terminates} and
Lemma~\ref{lem:alfa>2_terminates_boundconnectedcomponent}. Its proof is in
Appendix~\ref{proof of lem:alfa>2_degreeandsuperbridges}.

\begin{lemma}
\label{lem:alfa>2_degreeandsuperbridges}
Assume the hypotheses of Lemma \ref{lem:alfa>2_terminates} (and Lemma 
\ref{lem:alfa>2_terminates_boundconnectedcomponent}), so that 
$\Mgra{G_p}=(\Vl,\Mgra{E_p})$ is the \lgraph \ of $G_p$ generated by any fixed interval partition of $V$. Then, for each 
$\mbox{\sc v} \in \Vl$, we have:
\begin{equation}
    \label{eq:lem:no_bridge_GM}
    \Prc{\deg_{\Mgra{G_p}}(\mbox{\sc v}) = 0} \geq (1-p)^2 e^{-2/(\alpha-2)},
\end{equation}
\begin{equation}\label{eq:lem:no_bridge_GM_longerthan1} 
	\Prc{\text{$\mbox{\sc v}$ is incident to a super-bridge in $\Mgra{G_p}$}}\leq 
	\frac{2}{(\alpha-2)\ell^{\alpha-2}}\, .
\end{equation}
\end{lemma}

\subsection{Proof of Lemma \ref{lem:alfa>2_terminates}}
\label{ssec:lem:alfa>2_terminates}

To prove Lemma~\ref{lem:alfa>2_terminates}, we consider
Algorithm~\ref{alg:BFS-visit} below that performs a BFS visit of the
$\ell$-graph of the percolation graph $G_p$. Then, in
Lemma~\ref{lem:termination_alg_alfa>2}, we prove that this algorithm terminates
after visiting (only) $\mathcal{O}(\log n)$ super-nodes, w.h.p., for a
sufficiently large $\ell=\mathcal{O}(1)$. Together with
Fact~\ref{fact:connectedcomp-GMandG}, this proves
Lemma~\ref{lem:alfa>2_terminates}.

\begin{algorithm}[]
\caption{BFS visit of  $\Mgra{G_p}$}
\small{
    \begin{algorithmic}[1]
    \State \textbf{Input}: The   $\ell$-graph   $\Mgra{G_p}=(\Vl,\Mgra{E_p})$    of $G_p$;  an initiator (super-node) $\mbox{\sc s}  \subseteq V^{(\ell)}$
	\State $Q = \{\mbox{\sc s}\}$
	\While{$Q \neq \emptyset $}\label{line:while:ellgraphs}
	    \State $\mbox{\sc w}  = \dequeue(Q)$
	    \State $\texttt{visited}(\mbox{\sc w})= \texttt{True}$
 	            \For {each $\mbox{\sc x}$ s.t. $(\mbox{\sc x},\mbox{\sc w} )$ is a super-edge of $\Mgra{E_p}$
 	            and $\texttt{visited}(x)= \texttt{False}$}\label{line:ellgraphs:cond1}
 	                       \State $\enqueue(\mbox{\sc x},Q)$\label{line:ellgraph:enqueuering}
 	                  \EndFor
 	                \For{each $\mbox{\sc y}$ s.t. $(\mbox{\sc y},\mbox{\sc w})$ is a super-bridge of $\Mgra{E_p}$
 	                and $\texttt{visited}(\mbox{\sc y}) =  \texttt{False}$}\label{line:ellgraphs:cond2}
 	                \State $\enqueue(\mbox{\sc y},Q)$\label{line:ellgraph:enqueuebridges}
 	           \State $\mbox{\sc y}_{\text{left}}=\text{the super-node at distance $1$ from $\mbox{\sc y}$ on the ring at its left}$
  \If{$\texttt{visited}(\mbox{\sc y}_{\text{left}})= \texttt{False}$ and $(\mbox{\sc y}_{\text{left}},\mbox{\sc y}) \in {E_p}^{(\ell)}$}\label{line:ellgraphs:cond3}
 	                \State $\enqueue(\mbox{\sc y}_{\text{left}},Q)$
 	                   \EndIf
 	                \EndFor
 	 \EndWhile
	\end{algorithmic}}
\label{alg:BFS-visit}
\end{algorithm}

\paragraph{Remark} Note that, each time we add a super-node to queue $Q$, we
also also add the super-node to its left on the ring. So, in each while loop at
line~\ref{line:while:ellgraphs} of the algorithm, for each super-node
$\mbox{\sc w} \in Q$, its left neighbor on the ring will also be in $Q$. This
implies that, in each while loop, a single super-node is added to the queue at
line~\ref{line:ellgraph:enqueuering}.

\begin{lemma}
\label{lem:termination_alg_alfa>2}
Assume the hypotheses of Lemma~\ref{lem:alfa>2_terminates} and fix any node $s
\in V$. For any fixed interval partition\footnote{According to
Definition~\ref{def:l-graph}.} of the vertex set $V$, consider the $\ell$-graph
$\Mgra{G_p}$ and the super-node $\mbox{\textsc s}\in\Vl$ such that $s\in
I_{\mbox{\textsc s}}$. Then, a sufficiently large $\ell=\mathcal{O}(1)$ exists,
depending only on $p$ and $\alpha$, such that Algorithm~\ref{alg:BFS-visit}
terminates within $\mathcal{O}(\log n)$ iterations of the while loop in
line~\ref{line:while:ellgraphs}, with probability at least $1-1/n^2$.
\end{lemma}

\begin{proof}
For $t=1,2,\dots$, let $Q_t$ contents of $Q$ at the end of the $t$-th iteration
of the while loop of Algorithm~\ref{alg:BFS-visit} and let $W_t$ denote the
subset of super-nodes that were added to $Q$ during the $t$-th iteration. We
have $|Q_0|=1$ and
\begin{equation*}
    |Q_t| = \begin{cases} 0 \text{ if $|Q_{t-1}|=0$} \\ |Q_{t-1}|+|W_t|-1 \text{ otherwise},
    \end{cases}
\end{equation*}

Let $X_t$ and $Y_t$ denote the sets of super-nodes that were added to $Q$ in
the $t$-th iteration of the while loop, respectively at
line~\ref{line:ellgraph:enqueuering} and at
line~\ref{line:ellgraph:enqueuebridges}.  For $t \geq 2$, whenever a super-node
is added to the queue, the queue contains a super-node at ring-distance $1$
from it on the ring $(V^{(\ell)},E_1^{(\ell)})$, so that $|W_t|= |X_t|+2|Y_t|$,
where $|X_t|\leq 1$ for $t \geq 2$ and $|X_1|\leq 2$\,. Let $\delta=1-p$,
$\epsilon=2-\alpha$ and note that, from~\eqref{eq:lem:no_bridge_GM}
and~\eqref{eq:lem:no_bridge_GM_longerthan1}, 

\begin{equation}
\label{eq:prop_Xt_Yt}
    \Prc{|X_t|=1} \leq 1-\delta^2 e^{-2/\epsilon} \quad \text{and} 
    \quad \Expcc{|Y_t|} \leq \frac{2}{\epsilon \ell^\epsilon}.
\end{equation}
Moreover, note that for every $t$, $|Q_t|=0$, whenever $\sum_{i=1}^t|W_i|\leq
t$. Then, we can write
\begin{equation*}
    \sum_{i=1}^t|W_i|-t = \sum_{i=1}^t |X_i| + 2\sum_{i=1}^{t}|Y_i|-t\,,
\end{equation*}
where $|X_i|$ are $\{0,1\}$ random variables and, if we $w$ is the node
extracted from $Q$ in the $i$-th iteration of the while loop, $|Y_i|$ is the
number of super-bridge neighbors of $\mbox{\textsc w}$, so that it can also be
written as the sum of random variables in $\{0,1\}$. Note that   $\{X_t\}_t$
and $\{Y_t\}_t$ are not independent random variables, because of the conditions
that appear in lines \ref{line:ellgraphs:cond1}, \ref{line:ellgraphs:cond2} and
\ref{line:ellgraphs:cond3} of Algorithm \ref{alg:BFS-visit}. Still, it is easy
to show that $X_t$'s and $Y_t$'s are dominated by independent copies of two
random variables $X$ and $Y$, such that $\Prc{|X|=1}=1-\delta^2
e^{-2/\epsilon}$ and $\Expcc{|Y|} \leq 1/(\epsilon \ell^\epsilon)$.  Hence,
Chernoff bound and \eqref{eq:prop_Xt_Yt} imply that, for a certain
$t=\mathcal{O}(\log n)$ depending on $\delta$ and $\epsilon$,
\begin{align*}
    &\Prc{\sum_{i=1}^t|X_i| \geq 
    \left(1-\frac{\delta^2}{2}e^{-2/\epsilon}\right)t}\leq 
    \frac{1}{2n^2},\\ 
    &\Prc{\sum_{i=1}^t|Y_i| \geq \frac{2t}{\epsilon \ell^\epsilon}}\leq 
    \frac{1}{2n^2}.
\end{align*}
As a result, with probability at least $1-1/n^2$, we have:
\[
	\sum_{i=1}^t|X_i| +2\sum_{i=1}^t|Y_i| -t \leq t\left(\frac{4}{\epsilon \ell^\epsilon} -\frac{\delta^2}{2}e^{-2/\epsilon}\right).
\]
Hence, we can choose $\ell=\mathcal{O}(1)$ (depending only on $\epsilon$ and
$\delta$) large enough, so that with probability at least $1-1/n^2$,
$\sum_{i=1}^t|W_i| < t$, so that $|Q_t|=0$ for $t=\mathcal{O}(\log n)$.
\end{proof}

\subsection{Proof of Lemma \ref{lem:alfa>2_terminates_boundconnectedcomponent}}
\label{ssec:lem:alfa>2_terminates_boundconnectedcomponent}

Lemma~\ref{lem:alfa>2_terminates_boundconnectedcomponent} follows from
Fact~\ref{fact:connectedcomp-GMandG} and Lemma~\ref{le:small_distance} below.

\begin{lemma}\label{le:small_distance}
Assume the hypotheses of
Lemma~\ref{lem:alfa>2_terminates_boundconnectedcomponent}. For sufficiently
large $\ell$ depending only on $p$ and $\alpha$, consider the $\ell$-graph
$\Mgra{G_p}$ and the super-node $\textsc{s} \in\Vl$ such that $s\in
I_{\textsc{s}}$. Then, with probability at least 
\[
1-\tfrac{8}{(\alpha-2)}\cdot  \ell^{-(\alpha-2)/2}\, ,
\]
all nodes in $\Gamma_{\Mgra{G_p}}(\textsc{s})$ are within ring-distance $\ell$
from $\textsc{s}$ in the ring $(\Mgra{V},\Mgra{E_1})$. 

\end{lemma}
\begin{proof}
Without loss of generality, we assume $\textsc{s}=0$ and let
$\Mgra{V}=\{0,\dots,n/\ell\}$\,. Next, we consider nodes on the ring at
increasing distance from $\textsc{s}=0$ moving counter-clockwise, proving that
nodes at distance exceeding $\ell^{(2-\alpha)/2}$ are not part of
$\Gamma_{\Mgra{G_p}}(\textsc{s})$, the connected component of $\textsc{s}$. To
this purpose, denote by $K$ be the random variable indicating the ring-distance
of the super-node $\textsc{v}$ closest to $\textsc{s}$, such that $\textsc{v}$
has no incident super-edges in $\Mgra{G_p}$.  From \eqref{eq:lem:no_bridge_GM},
and setting $\epsilon=2-\alpha$ and $\delta=1-p$, we have
\[
	\Prc{\text{\textsc{v} has no super-edges in $\Mgra{G_p}$}}\geq \delta^2 e^{-2/\epsilon},
\]
So, for $k=\ell^{\epsilon/2}$ and for $\ell$ large enough
\[
	\Prc{K >k} \leq (1-\delta^2 e^{-2/\epsilon})^k \leq \frac{1}{\ell^{\epsilon/2}}\,.
\]
Moreover, denote by $B_k$ be the following event
\[
B_k=\{\text{the $k$ nodes nearest to $\textsc{s}$ have no super-bridges}\}.
\]
From~\eqref{eq:lem:no_bridge_GM_longerthan1}, from the independence of the edge
percolation events and using a union bound, we have for the complementary event
$B_k^C$
\[
	\Prc{B_k^C}\leq k \cdot \frac{2}{\epsilon \ell^\epsilon}\,.
\]
Iterating the same argument for nodes on the clockwise side of 
$\textsc{s}$, if $k=\ell^{\epsilon/2}$ we have
\begin{align*}
    &\Prc{\text{there is a node at distance $\leq k$ from $\textsc{s}$ in $\Mgra{\Gamma}_p(\textsc{s})$}} \\ 
    &\leq 2\Prc{\{K>k\} \cup B_k^C} \leq \frac{4}{\ell^{\epsilon/2}}+\frac{4}{\epsilon \ell^{\epsilon/2}},
\end{align*}
which completes the proof.
\end{proof}

\section{The case $1< \alpha < 2$}
\label{sec:1<alpha<2}

The aim of this section is to prove Theorem~\ref{thm:main_1alpha2}, which
follows from the two lemmas below.

\begin{lemma}
\label{lem:1alpha2_underthreshold}
Under the hypotheses of Theorem \ref{thm:main_1alpha2}, assume $G=(V,E)$ is
sampled from \SWGpl and let $G_p$ be its percolation graph. Then two constants
$\overline{p}<1$ and $\eta>1$ exist such that, if $p>\overline{p}$, w.h.p.
there is a set of $\Omega(n)$ nodes in $G_p$ that induces a connected
sub-component with diameter $\bigO(\log^\eta n)$.
\end{lemma}

\begin{lemma}
\label{lem:1alpha2_belowthreshold}
Under the hypotheses of Theorem \ref{thm:main_1alpha2}, assume $G=(V,E)$ is
sampled from \SWGpl and let $G_p$ its percolation graph.  Then, a constant
$\underline{p}>0$ exists  (in particular, $\underline{p}=1/3$) such that, if
$p<\underline{p}$, w.h.p. each connected component of $G_p$ has size at most
$\bigO(\log n)$.
\end{lemma}

The proof of Lemma~\ref{lem:1alpha2_underthreshold} is given in
Subsection~\ref{ssec:lem:1alpha2_underthreshold}, while the proof of
Lemma~\ref{lem:1alpha2_belowthreshold} follows an  approach, based on
Galton-Watson processes, similar to that of Lemma~\ref{lem:alfa>2_terminates},
and it is given in Appendix~\ref{ssec:lem:1alpha2_belowthreshold}. 

\subsection{Proof of Lemma~\ref{lem:1alpha2_underthreshold}}
\label{ssec:lem:1alpha2_underthreshold}
The lemma that follows is our new key ingredient to  apply an inductive
approach similar to the one in \cite{BenjaminiB01} in the case of full-bond
percolation.

\begin{figure*}[]
    \centering
    \includegraphics[width=0.7\textwidth]{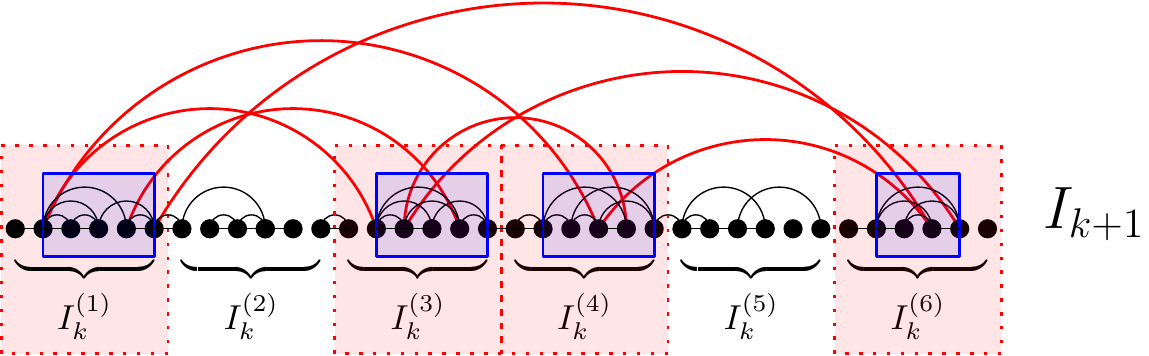}
    
    \caption{Visualization of the proof of
    Lemma~\ref{lem:1alpha2_underthreshold}. Red dotted boxes identify ``good''
    intervals in $H_{k+1}$ (intervals for which the event $A_k$ holds); for
    each good box, the included blue box contains the subset of nodes that
    induce a connected subgraph of size at least $(1-\varepsilon_k)N_k$ nodes
    and diameter $D_k$. In the lemma, we have to prove the existence of at
    least a certain number of red boxes, and that blue boxes are all mutually
    connected (i.e. the existence of the red edges in the picture). }
    
    \label{fig:renormalization}
\end{figure*}

\begin{lemma}\label{lem:key_1alpha2}
Assume the hypotheses of Lemma \ref{lem:1alpha2_underthreshold}, let $\beta =
\frac{1}{2}\alpha(3-\alpha)$ and
\[
N_k = e^{\beta^k}\quad \text{and} \quad C_k = e^{\beta^{k-1}(\beta-1)}.
\]
Let $I_k$ be an arbitrary interval of $N_k$ adjacent nodes in the cycle
$(V,E_1)$, $D_k$ any finite integer. Define the following event:

\begin{equation*}
	A_k = \{\exists S\subseteq I_k: |S|\geq 
	(1-\varepsilon_k)N_k\ \wedge \diam_p(S)\le D_k\}.
\end{equation*}
Assume further that, for suitable, real constants $\varepsilon_k,\delta_k$ and
$p_k$ in $(0,1)$,

\begin{equation*}
    \Prc{A_k} \geq 1-\delta_k, \quad \text{if $p\geq p_k$}.
\end{equation*}
Then, if we consider an interval $I_{k+1}$ of $N_{k+1}$ adjacent nodes, it holds:  
\[
\Prc{A_{k+1}}\geq 1-\delta_{k+1}, \quad \text{ if $p \geq p_{k+1}$},
\]
where
\begin{align*}
	&\delta_{k+1} = 2C_{k+1}^{-0.2}, \quad \varepsilon_{k+1} = 
	\varepsilon_k+\delta_k+C_{k+1}^{-0.2}, \\
	& D_{k+1} = 2D_{k}+1, 
	\quad p_{k+1} = 
	\frac{0.9(\alpha-1)(2-\alpha)}{(1-\varepsilon_k)^2(4.2-2\alpha)}C(n,\alpha). 
\end{align*}
\end{lemma}

\begin{proof}
In words, $A_k$ is the event that there exists a subset of the nodes in $I_k$
that induces a connected component of $G_p$ of size $\geq (1-\varepsilon_k)N_k$
nodes and diameter $\leq D_k$. Consider the interval $I_{k+1}$ and divide it
into $C_{k+1}$ disjoint intervals of size $N_k$ (note that $N_{k+1} = N_k \cdot
C_{k+1}$). Denote by $H_{k+1}$ the set of sub-intervals of $I_{k+1}$ of size
$N_{k}$ for which event $A_k$ holds. In particular, for every $i = 1, 2, \dots,
C_{k+1}$, we use the indicator variable $A_k^{(i)}$ to specify whether or not
$A_k$ holds for the $i$-th sub-interval, so that 
\[
	|H_{k+1}|= \sum_{i=1}^{C_{k+1}}A_{k}^{(i)}.
\]
where independence of the $A_k^{(i)}$'s follows from disjointness of the
sub-intervals of $I_{k+1}$.  For the sake of brevity, let $B_{k+1}$ denote the
event $\{ |H_{k+1}| \geq (1-\delta_k) C_{k+1} - C_{k+1}^{0.8}\}$. Application
of Chernoff's bound then implies
\begin{equation*}
    \Prc{B_{k+1}}\geq 1- e^{-2C_{k+1}^{0.6}}\geq 1-C_{k+1}^{-0.2}.
\end{equation*}

Note that by definition, each sub-intervals in $H_{k+1}$ contains at least one
subset (of the nodes) that induces a connected component in $G_p$ with at least
$(1-\varepsilon_k)N_k$ nodes and diameter at most $D_k$. Our goal is to show
that, with some probability, these intervals are all connected to each other:
this in turn implies the existence of a set of nodes that induces a larger
connected component, with diameter at most $2D_k +1$ and containing at least
$|H_{k+1}|(1-\varepsilon_k)N_k$ nodes, see see Figure~\ref{fig:renormalization}
for a visual intuition of the proof. In the remainder of this proof, we denote
by $F_{k+1}$ the event that all connected components associated to intervals in
$H_{k+1}$ are mutually connected. In particular, we prove that $\Prc{F_{k+1}
\mid B_{k+1}} \geq 1-C_{k+1}^{-0.2}$, so that

\begin{align*}
	&\Prc{A_{k+1}}=\Prc{F_{k+1} \cap B_{k+1}}= 
	\Prc{B_{k+1}}\Prc{F_{k+1} \mid B_{k+1}}\\
	&\geq 1-2C_{k+1}^{-0.2},
\end{align*}
which implies that $A_{k+1}$ holds with probability at least $1-\delta_{k+1}$,
whenever we set 
\[
\delta_{k+1} =2C_{k+1}^{-0.2} \quad \varepsilon_{k+1} = \varepsilon_k 
+ \delta_k + C_{k+1} ^ {-0.2} \quad \text{and} \quad D_{k+1} = 2D_k + 
1.
\]
Now, we estimate the probability that, given $B_{k+1}$, $F_{k+1}$ holds. Two
nodes in $I_{k+1}$ have distance at most $N_{k+1}$ on the cycle. Moreover, if
we consider two sub-intervals of $I_{k+1}$ both belonging to $H_{k+1}$, the
corresponding connected components contain at least $(1-\varepsilon_k) N_k$
nodes each, accounting for at least $(1-\varepsilon_k)^2 N_k^2$ pairs $(u, v)$,
with $u$ belonging to the first and $v$ to second connected component. So, two
given sub-intervals in $H_{k+1}$, they are not connected with probability at
most
\begin{align*}
    &\left(1-\frac{p}{c 
    N_{k+1}^\alpha}\right)^{(1-\varepsilon_k)^2N_k^2} \leq \exp 
    \left(-\frac{p}{c}(1-\varepsilon_k)^2 \frac{e^{2\beta^k}}{e^{\alpha \beta^{k+1}}}\right)\\ 
    &= \exp \left(-\frac{p}{c}(1-\varepsilon_k)^2 e^{\beta^{k}(2-\alpha 
    \beta)}\right),
\end{align*}
where in the remainder of this proof, we write $c$ for $C(\alpha, n)$, for the
sake of readability. If we consider all pairs of intervals in $H_{k+1}$, a
simple union bound allows us to conclude that the probability that the
intervals in  $H_{k+1}$ are not all mutually connected is at most
\begin{align*}
	&C_{k+1}^2\exp\left(-\frac{p}{c}(1-\varepsilon_{k})^2e^{\beta^k(2-\alpha\beta)}\right)\\
	&\leq \exp \left(\beta^k \left(2(\beta-1)-\frac{p}{c}(1-\varepsilon_k)^2 (2(2-\alpha)+0.2) \right)\right),
\end{align*}
where the last inequality follows from the definition of $\beta$.  Finally, the
quantity above can be upper bounded as follows
\begin{align*}
  \exp \left(-\beta^k 0.1(\alpha-1)(2-\alpha)\right)=\exp(-0.2\beta^k(\beta-1))=C_{k+1}^{-0.2},
\end{align*}
whenever $p$ satisfies
\begin{equation*}
	p \geq \frac{c(2-\alpha)(\alpha-1)0.9}{(1-\varepsilon_k)^2(2(2-\alpha)+0.2)}.
\end{equation*}
\end{proof}

Now we are ready to prove the lemma.

\begin{proof}[Proof of Lemma \ref{lem:1alpha2_underthreshold}]
Let $C_k$ and $N_k$ defined as in the claim of Lemma~\ref{lem:key_1alpha2}.
First, we consider the series $\sum_{k}C_k^{-0.2}$ and, since
$C_k=e^{\beta^{k-1}(\beta-1)}$ with $\beta>1$, we notice that the series is
convergent, i.e.  $\sum_{k=1}^{\infty}C_k^{-0.2}<+\infty$.\footnote{We did not
try to optimize constants and the choice $0.2$ for the exponent is not
necessarily optimal.} This means that the tail of the series converges to zero,
and this implies that there exists a constant $h$ such that
\begin{equation}
    \sum_{k=h}^{+\infty} C_k^{-0.2}\leq \frac{1}{100}.
    \label{eq:tail_ck}
\end{equation}
The constant $h$ depends only on $\alpha$. In particular, it increases 
as $\alpha \rightarrow 2^-$ and $\alpha \rightarrow 1^+$.

Now we consider $I_h$, an arbitrary interval of size $N_h = 
e^{\beta^h}$. Next, we let
\begin{equation*}
    \delta_h = 1-p^{e^{\beta^h}}, \quad \varepsilon_h 
             = 0, \quad D_h = e^{\beta^h}, \quad p_h = p,
\end{equation*}
and we consider the event $A_h$, defined as in the statement of
Lemma~\ref{lem:key_1alpha2}.  If no ring edge belonging to $I_h$ is percolated,
$A_h$ is trivially true: this implies that, for every $p\geq p_h$, $\Prc{A_h}
\geq 1-\delta_h,$ where probability is over the edges with endpoints in $I_h$.
If $A_{h+1}$ and $N_{h+1}$ are defined like in its statement,
Lemma~\ref{lem:key_1alpha2} then implies that, for an arbitrary interval
$I_{h+1}$ of size $N_{h+1}$, we have
\[
\Prc{A_{h+1}} \geq 1-\delta_{h+1},
\]
if $p \geq p_{h+1}$ and whenever we take

\begin{align*}
    &\delta_{h+1} =2C_{h+1}^{-0.2}, \quad \varepsilon_{h+1} = 
    1-p^{e^{\beta^h}} + C_{h+1}^{-0.2},\\
    &D_{h+1} = 2e^{\beta^h}+1, \quad p_{h+1} = 
    \frac{0.9c(\alpha-1)(2-\alpha)}{(4.2-2\alpha)},
\end{align*}
where the probability is taken over the randomness of the edges with endpoints
in $I_{h+1}$.\footnote{Recall that $c=C(\alpha, n)$ in the remainder of this
proof.}

If we iteratively apply Lemma \ref{lem:key_1alpha2}, we thus have for each
$k\geq 1 $ such that $N_k \leq n$,
\begin{equation*}
    \Prc{A_k} \geq 1-\delta_k,
\end{equation*}
where $\delta_k =2C_{k}^{-0.2}$, $p_k =
\frac{0.9c(\alpha-1)(2-\alpha)}{(1-\varepsilon_{k-1})^2(4.2-2\alpha)}$, and
$\varepsilon_k$ and $D_k$  are defined by the recurrencies:
\begin{equation*}
    \begin{cases}
    \varepsilon_k =\varepsilon_{k-1}+\delta_{k-1}+C_{k}^{-0.2},\ 
    \text{if $k > h$} \\ 
    \varepsilon_{h} = 1-p^{e^{\beta^h}}
    \end{cases}
\end{equation*}

\begin{equation*}
    \begin{cases}
    D_k = 2D_{k-1}+1,\ \text{if $k > h$} \\
    D_h = e^{\beta^h}.
    \end{cases}
\end{equation*}
Now, we solve the recurrence for $\varepsilon_k$, obtaining $ \varepsilon_k =
1-p^{e^{\beta^h}} + 3\sum_{i=h}^k C_i ^{-0.2},$ where, leveraging
\eqref{eq:tail_ck}, we take $p>\sigma$, with $\sigma 
> 0$ such that $1-\sigma^{e^{\beta^h}}=1/50$. With this choice, for 
each $k \geq 1$ we obtain
\begin{equation*}
    \varepsilon_k \leq 1-p^{e^{    \beta^h}}+\frac{3}{100} \leq \frac{1}{20}.
\end{equation*}
Moreover, for each $i \leq k$, we have that $D_k \leq  2^{k - h + 1}
e^{\beta^h}$. If we take $m=\log_\beta(\log n)$, then $N_m=e^{\beta^m}=n$ and
$C_m = n^{(\beta-1)/\beta}$. We also have:
\begin{align*}
	&D_m \leq \frac{e^{\beta^h}}{2^{h-1}}2^{\log_\beta(\log(n))}
	=\frac{e^{\beta^h}}{2^{h-1}}(\log n)^{\log_\beta 2}.
\end{align*}
Setting $\eta=\log_\beta 2$, we have $\eta > 1$, since $\beta<2$. We thus have
\[
D_m = O((\log^\eta(n))).
\] 
Moreover, if $p \geq \sigma$
\begin{equation*}
    \delta_m \leq 2n^{-0.2(\beta-1)/\beta},\ \varepsilon_m \leq 
    \frac{1}{20}, \ p_m \geq \frac{0.9c(\alpha-1)(2-\alpha)}{(19/20)^2(4.2-2\alpha)}.
\end{equation*}
Finally, if
\[p \geq \max\left\{\sigma, \frac{0.9c(\alpha-1)(2-\alpha)}{(19/20)^2(4.2-2\alpha)}\right\}:=\overline{p}\]
then,
\[\Prc{A} \geq 1-\delta_m \geq 1-2n^{-0.2(\beta-1)/\beta},\]
where
\[
	A = \left\{\exists S\subseteq V: |S| > \frac{19}{20}n\,\wedge\, 
	\diam_{p}(S) = O(\log^\eta (n))\right\},
\]
i.e., w.h.p. $G_p$ contains an induced subgraph of size at least $(19/20)n$ 
nodes and diameter $O(\log^\eta (n))$.
\end{proof}

\section{The case $\alpha<1$}
\label{sec:alpha<1}
In this section, we prove Theorem~\ref{thm:main_1alpha}, which follows from the
two lemmas below.

\begin{lemma}\label{le:1alpha_low}
Under the hypotheses of Theorem~\ref{thm:main_1alpha}, sample a $G=(V,E)$ from
the \SWGpl distribution and let $G_p$ be its  percolation subgraph. Then, a
constant $\overline{p}<1$ exists such that, if $p > \overline{p}$ then, w.h.p.,
there is a set of $\Omega(n)$ nodes in $G_p$ that induce a connected subgraph
of diameter $\bigO(\log n)$.
\end{lemma}

\begin{lemma}\label{le:1alpha_up}
Under the hypothesis of Theorem~\ref{thm:main_1alpha}, sample a $G=(V,E)$ from
the \SWGpl distribution and let $G_p$ be the percolation subgraph of $G$. Then,
a constant $\underline{p}>0$ exists such that, if $p<\underline{p}$ then,
w.h.p., each connected component of $G_p$ has size at most $\bigO(\log n)$.
\end{lemma}

A key observation to prove Lemma~\ref{le:1alpha_low} is that, when $\alpha<1$,
the percolation graph of a graph sampled from \SWGpl stochastically dominates a
cycle with additional \Erdos\ random edges. To prove Lemma~\ref{le:1alpha_low}
we thus use a previous result in~\cite{becchettiCDPTZ21} for this class of
random graphs. In particular, we first give an equivalent formulation of
Lemma~5.1 in~\cite{becchettiCDPTZ21}, stating that, with constant probability
the sequential BFS visit (Algorithm~\ref{alg:sequential-BFS-visit}) reaches
$\Omega(\log n)$ nodes within $\bigO(\log n)$ rounds. We then consider a
\textit{parallel} BFS visit (see Algorithm~\ref{alg:parallel-BFS-visit} in
Subsection~\ref{ssec:1alpha_low} of the Appendix) and show an equivalent
formulation of Lemma~5.2 in~\cite{becchettiCDPTZ21}, stating that, w.h.p., the
parallel BFS visit starting with $\Omega(\log n)$ nodes reaches a constant
fraction of nodes within $\bigO(\log n)$ rounds.

The full proof of Lemma~\ref{le:1alpha_low} is given in
Subsection~\ref{ssec:1alpha_low} of the Appendix, while the proof of
Lemma~\ref{le:1alpha_up} is omitted, since it proceeds along the very same
lines as the proof of Lemma~\ref{lem:1alpha2_belowthreshold}, which does not
depends on the value of $\alpha$.

\bibliographystyle{plain}
\bibliography{njl}

\begin{thebibliography}{10}

\bibitem{ABS04}
Noga Alon, Benjamini Itai, and Stacey Alan.
\newblock Percolation on finite graphs and isoperimetric inequalities.
\newblock {\em Annals of Probability}, 32:1727--1745, 2004.

\bibitem{Barab509}
Albert-L{\'a}szl{\'o} Barab{\'a}si and R{\'e}ka Albert.
\newblock Emergence of scaling in random networks.
\newblock {\em Science}, 286(5439):509--512, 1999.

\bibitem{becchettiCDPTZ21}
Luca Becchetti, Andrea Clementi, Riccardo Denni, Francesco Pasquale, Luca
  Trevisan, and Isabella Ziccardi.
\newblock Sharp thresholds for a sir model on one-dimensional small-world
  networks, 2021.

\bibitem{BenjaminiB01}
Itai Benjamini and Noam Berger.
\newblock The diameter of long-range percolation clusters on finite cycles.
\newblock {\em Random Struct. Algorithms}, 19(2):102--111, 2001.

\bibitem{BBLR12}
Itai Benjamini, Stephane Boucheron, Gabor Lugosi, and Raphael Rossignol.
\newblock {Sharp threshold for percolation on expanders}.
\newblock {\em The Annals of Probability}, 40(1):130 -- 145, 2012.

\bibitem{Bis04}
Marek Biskup.
\newblock {On the scaling of the chemical distance in long-range percolation
  models}.
\newblock {\em The Annals of Probability}, 32(4):2938 -- 2977, 2004.

\bibitem{Bis11}
Marek Biskup.
\newblock Graph diameter in long-range percolation.
\newblock {\em Random Structures \& Algorithms}, 39(2):210--227, 2011.

\bibitem{Bollo07}
B\'ela Bollob\'as, Svante Janson, and Oliver Riordan.
\newblock The phase transition in inhomogeneous random graphs.
\newblock {\em Random Structures \& Algorithms}, 31(1):3--122, 2007.

\bibitem{Bollobs2004TheDO}
B{\'{e}}la Bollob{\'{a}}s and Oliver Riordan.
\newblock The diameter of a scale-free random graph.
\newblock {\em Comb.}, 24(1):5--34, 2004.

\bibitem{CDHKNS01}
Duncan~S. Callaway, John~E. Hopcroft, Jon Kleinberg, Mark E.~J. Newman, and
  Steven~H. Strogatz.
\newblock Are randomly grown graphs really random?
\newblock {\em Phys. Rev. E}, 64:041902, Sep 2001.

\bibitem{CWLC13}
Wei Chen, Laks Lakshmanan, and Carlos Castillo.
\newblock Information and influence propagation in social networks.
\newblock {\em Synthesis Lectures on Data Management}, 5:1--177, 10 2013.

\bibitem{CPGE19}
Hyeongrak Choi, Mihir Pant, Saikat Guha, and Dirk Englund.
\newblock Percolation-based architecture for cluster state creation using
  photon-mediated entanglement between atomic memories.
\newblock {\em npj Quantum Information}, 5(1):104, 2019.

\bibitem{EK10}
Easley David and Kleinberg Jon.
\newblock {\em Networks, Crowds, and Markets: Reasoning About a Highly
  Connected World}.
\newblock Cambridge University Press, USA, 2010.

\bibitem{DGM06}
Moez Draief, Ayalvadi Ganesh, and Laurent Massouli\'{e}.
\newblock Thresholds for virus spread on networks.
\newblock In {\em Proceedings of the 1st International Conference on
  Performance Evaluation Methodolgies and Tools}, valuetools '06, page 51–es,
  New York, NY, USA, 2006. Association for Computing Machinery.

\bibitem{dubhashipanconesi09}
Devdatt Dubhashi and Alessandro Panconesi.
\newblock {\em Concentration of Measure for the Analysis of Randomized
  Algorithms}.
\newblock Cambridge University Press, USA, 1st edition, 2009.

\bibitem{Durrett1990ATT}
Rick Durrett and Harry Kesten.
\newblock The critical parameter for connectedness of some random graphs.
\newblock {\em A Tribute to P. Erdos}, pages 161--176, 1990.

\bibitem{GGT03}
Michele Garetto, Weibo Gong, and Donald~F. Towsley.
\newblock Modeling malware spreading dynamics.
\newblock In {\em Proceedings {IEEE} {INFOCOM} 2003, The 22nd Annual Joint
  Conference of the {IEEE} Computer and Communications Societies, San Franciso,
  CA, USA, March 30 - April 3, 2003}, pages 1869--1879. {IEEE} Computer
  Society, 2003.

\bibitem{KNT94}
Anna~R. Karlin, Greg Nelson, and Hisao Tamaki.
\newblock On the fault tolerance of the butterfly.
\newblock In {\em Proceedings of the twenty-sixth annual ACM symposium on
  Theory of Computing}, pages 125--133, 1994.

\bibitem{KKT15}
David Kempe, Jon Kleinberg, and \'{E}va Tardos.
\newblock Maximizing the spread of influence through a social network.
\newblock {\em Theory of Computing}, 11(4):105--147, 2015.

\bibitem{KetAl80}
Harry Kesten.
\newblock The critical probability of bond percolation on the square lattice
  equals 1/2.
\newblock {\em Communications in mathematical physics}, 74(1):41--59, 1980.

\bibitem{K00}
Jon Kleinberg.
\newblock Navigation in a small world.
\newblock {\em Nature}, 406:845, 2000.

\bibitem{KL19}
Alexander Kott and Igor Linkov.
\newblock {\em Cyber resilience of systems and networks}.
\newblock Springer, 2019.

\bibitem{LKAK16}
Eun~Jee Lee, Sudeep Kamath, Emmanuel Abbe, and Sanjeev~R. Kulkarni.
\newblock Spectral bounds for independent cascade model with sensitive edges.
\newblock In {\em 2016 Annual Conference on Information Science and Systems,
  {CISS} 2016, Princeton, NJ, USA, March 16-18, 2016}, pages 649--653. {IEEE},
  2016.

\bibitem{LSV14}
Remi Lemonnier, Kevin Scaman, and Nicolas Vayatis.
\newblock Tight bounds for influence in diffusion networks and application to
  bond percolation and epidemiology.
\newblock In {\em Advances in Neural Information Processing Systems},
  volume~27, pages 846--854. Curran Associates, Inc., 2014.

\bibitem{LRSV14}
R\'{e}mi Lemonnier, Kevin Seaman, and Nicolas Vayatis.
\newblock Tight bounds for influence in diffusion networks and application to
  bond percolation and epidemiology.
\newblock In {\em Proceedings of the 27th International Conference on Neural
  Information Processing Systems - Volume 1}, NIPS'14, page 846–854,
  Cambridge, MA, USA, 2014. MIT Press.

\bibitem{moore2000epidemics}
Cristopher Moore and Mark E.~J. Newman.
\newblock Epidemics and percolation in small-world networks.
\newblock {\em Physical Review E}, 61(5):5678, 2000.

\bibitem{MCN00}
Cristopher Moore and Mark E.~J. Newman.
\newblock Exact solution of site and bond percolation on small-world networks.
\newblock {\em Phys. Rev. E}, 62:7059--7064, Nov 2000.

\bibitem{newman1986one}
Charles~M. Newman and Lawrence~S. Schulman.
\newblock One dimensional $1/|j-i|^s$ percolation models: The existence of a
  transition for $s \leqq 2$.
\newblock {\em Communications in Mathematical Physics}, 104(4):547--571, 1986.

\bibitem{newman1999scaling}
Mark E.~J. Newman and Duncan~J. Watts.
\newblock Scaling and percolation in the small-world network model.
\newblock {\em Physical review E}, 60(6):7332, 1999.

\bibitem{VespietAl15}
Romualdo Pastor-Satorras, Claudio Castellano, Piet Van~Mieghem, and Alessandro
  Vespignani.
\newblock Epidemic processes in complex networks.
\newblock {\em Rev. Mod. Phys.}, 87:925--979, Aug 2015.

\bibitem{VK71}
Vinod~K.S. Shante and Scott Kirkpatrick.
\newblock An introduction to percolation theory.
\newblock {\em Advances in Physics}, 20(85):325--357, 1971.

\bibitem{Wang_2017}
Wei Wang, Ming Tang, H~Eugene Stanley, and Lidia~A Braunstein.
\newblock Unification of theoretical approaches for epidemic spreading on
  complex networks.
\newblock {\em Reports on Progress in Physics}, 80(3):036603, feb 2017.

\bibitem{watts1998collective}
Duncan~J. Watts and Steven~H. Strogatz.
\newblock Collective dynamics of ‘small-world’ networks.
\newblock {\em nature}, 393(6684):440--442, 1998.

\bibitem{WSZH15}
Tongfeng Weng, Michael Small, Jie Zhang, and Pan Hui.
\newblock L{\'e}vy walk navigation in complex networks: A distinct relation
  between optimal transport exponent and network dimension.
\newblock {\em Scientific Reports}, 5(1):17309, 2015.

\end{thebibliography}

\newpage
\appendix
\begin{center}
	\LARGE{\textbf{Appendix}}
\end{center}

\section{Full-bond Percolation vs Independent Cascade} \label{sec:equi-bpic}

We first give the formal definitions of full-bond percolation and independent
cascade processes.

\begin{definition}[Full-bond percolation]
\label{def:live-arc-model}
Given a graph $G = (V,E)$, and given, for every edge $e \in E$, a
\emph{percolation probability} $p(e) \in [0,1]$, the \emph{bond percolation}
process consists to \emph{remove} each edge $e \in E$, independently, with
probability $1-p(e)$.  The random subgraph, called the \emph{percolation graph}
$G_p=(V,E_p)$, is defined by the edges that are not removed (i.e. they are
activated), i.e., $E_p = \{ e \in E \, : \, \mbox{edge $e$ is not removed} \}$.
Given an initial subset $A_0 \subseteq V$ of \emph{active} nodes, for every
integer $t \leq n-1$, we define the random subset $A_{t}$ of  $t$-active nodes
as the subset of nodes that are at distance $t$ from $A_{0}$ in the percolation
graph $G_p=(V,E_p)$, i.e., $A_{t} = N_{G_{p}}^{t}(A_{0})$. Finally, the subset
of all active nodes from $A_{0}$ is the subset $N^*_{G_p}(A_0)$.
\end{definition}

\begin{definition}[Independent cascade] \label{def:independent_cascade}
Given a graph $G = (V,E)$, an assignment of \emph{transmission probabilities}
$\{ p(e) \}_{e\in E}$ to the edges of $G$, and a non-empty set $I_0 \subseteq
V$ of initially infectious nodes, the {\em Independent Cascade} (for short,
\IC) protocol defines the stochastic process $\{S_t,I_t,R_t\}_{t\geq 0}$ on
$G$, where $S_t,I_t,R_t$ are three sets of vertices, respectively called
\emph{susceptible}, \emph{infectious}, and \emph{recovered}, which form a
partition of $V$ and that are defined as follows.

\begin{itemize}
     \item At time $t=0$ we have $R_0 = \emptyset$ and $S_0 = V-I_0$.
     \item At time $t\geq 1$:
     \begin{itemize}
         \item $R_t = R_{t-1} \cup I_{t-1}$, that is, the nodes that were
         infected at the previous step become recovered;
         
         \item independently for each edge $e = \{ u,v\}$ such that $u \in
         I_{t-1}$ and $v\in S_{t-1}$, with probability $p(e)$ the event that
         ``$u$ transmits the infection (i.e. a copy of the source message) to
         $v$ at time $t$'' takes place. The set $I_t$ is the set of all
         vertices $v\in S_{t-1}$ such that for at least one neighbor $u \in
         I_{t-1}$ the event that $u$ transmits the infection to $v$ takes place
         as described above.
         
         \item $S_t = S_{t-1} - I_t$
     \end{itemize}
 \end{itemize}
The process stabilizes when $I_t = \emptyset$.\\ The Reed-Frost \SIR protocol
(for short, \emph{RF Protocol}) is the special case of the \IC protocol in
which all transmission probabilities are the same. 
\end{definition}

In~\cite{KKT15}, given any fixed graph $G = (V,E)$, the Independent Cascade
process is shown to be \emph{equivalent} to the full-bond percolation process. 

If we consider the set $I_t$ of nodes that are infectious at time $t$ in a
graph $G=(V,E)$ according to the IC process with transmission probabilities $\{
p(e) \}_{e\in E} $ and with initiator set $I_0$, we see that such a set has
precisely the same distribution as the set of nodes at distance $t$ from $I_0$
in the percolation graph $G_p$ generated by the bond-percolation process with
probabilities $\{ p(e) \}_{e\in E}$ (see Definition \ref{def:live-arc-model}).
Furthermore, the set of recovered nodes $R_t$ is distributed precisely like the
set of nodes at distance $< t$ from $I_0$ in $G_p$.

We formalize this equivalence by quoting a theorem from~\cite{CWLC13}.

\begin{theorem}[Bond percolation and IC processes are equivalent, \cite{KKT15}]
\label{thm:equivalence}
Consider the bond-percolation process and the \IC protocol on the same graph
$G=(V,E)$ and let $I_0 = A_0 = V_0$, where $V_0$ is any fixed subset of $V$,
and with transmission probabilities and percolation probabilities equal to $\{
p(e) \}_{e\in E}$. 
 
Then, for every integer $t \geq 1$ and for every subsets $V_1, \dots, V_{t-1}
\subseteq V$, the events $\{I_0 = V_0, \dots, I_{t-1} = V_{t-1}\}$ and $\{A_0 =
V_0, \dots, A_{t-1} = V_{t-1}\}$ have either both zero probability   or
non-zero probability, and, in the latter case, the  distribution of the
infectious set $I_{t}$,  conditional to the event $\{I_0 = V_0, \dots, I_{t-1}
= V_{t-1}\}$, is   the same  to that of  the $t$-active set $A_t$, conditional
to the event  $\{A_0 = V_0, \dots, A_{t-1} = V_{t-1}\}$. 
\end{theorem}

The strong equivalence shown in the previous theorem is obtained by applying
the principle of deferred decision on the percolation/infection events that
take place on every edge since they are mutually independent. This result can
be exploited to analyze different aspects and issues of the IC (and, thus, the
RF) process. We here summarize such aspects in an informal way.

As a first immediate consequence of Theorem~\ref{thm:equivalence}, we have
that, starting from any source subset $I_0$, to bound the size of the final set
$R_{\tau}$ of the nodes informed by $I_0$, we can look at the size of the union
of the connected components in $G_p$ that include all nodes of $I_0$, i.e., we
can bound the size of $N^*_{G_p}(I_0)$.

A further remark is that in the bond-percolation process there is no {\em
time}, and we can analyze the connected component of the percolation graph in
any order and according to any visit process. Furthermore, if we want a lower
bound to the number of nodes reachable from $I_0$ in the percolation graph, we
can choose to focus only on vertices reachable through a subset of all possible
paths, and, in particular, we can restrict ourselves to paths that are easier
to analyze.  In our analysis we will only consider paths that alternate between
using a bounded number of local edges and one bridge edge.

\section{Omitted proofs}

\subsection{Proof of Lemma \ref{lem:alfa>2_degreeandsuperbridges}}
\label{proof of lem:alfa>2_degreeandsuperbridges}

We first prove the following, preliminary fact.

\begin{lemma}
Let $n \geq 3$ and $\alpha>2$ and let $G=(V,E_1 \cup E_2)$ be a \SWGpl graph.
Then, for every $v \in V$ and $x \geq 0$, we have:
\begin{equation*}
    \Prc{\text{$v$ has a bridge of length $> x$}}\leq 
    \frac{1}{x^{\alpha-1}}\, ,
\end{equation*}
\begin{equation*}
    \Prc{\text{$v$ has a bridge with length $> x$ on one side}}\leq \frac{1}{2x^{\alpha-1}}\,.
\end{equation*}
\label{lem:length_bridges}
\end{lemma}

\begin{proof}
We have
\begin{align*}
    &\Prc{\text{$v$ has a bridge with length $>x$ in one side}}\leq 
    \sum_{y=x+1}^{n/2}\frac{1}{C(\alpha,n)y^{\alpha}}\\
    &\leq \int_{x+1}^{+\infty}\frac{1}{C(\alpha,n)(y-1)^{\alpha}}dy
    \leq \frac{1}{(\alpha-1)C(\alpha,n)x^{\alpha-1}}\leq \frac{1}{x^{\alpha-1}},
\end{align*}
where the last inequality follows from the fact that $C(\alpha,n)\geq 1/2^{\alpha-1}$. 
Then, the  lemma follows from 
\begin{align*}
	&\Prc{\text{$v$ has a bridge of length $>x$ }}\\
	&\leq 2\Prc{\text{$v$ has a bridge with length $>x$ on one side}}. 
\end{align*}
\end{proof}

Consider now any super-node $\mbox{\sc v} \in\Mgra{V}$ and, for the rest of
this proof, denote by $w_1$ and $w_2$ the two boundary nodes of $I_{\mbox{\sc
v}}$.  Without loss of generality, we assume $I_{\mbox{\sc v}}$ includes the
ring-edges that we traverse if we move on the cycle $(V, E_1)$ from $w_1$ to
$w_2$  counter-clockwise.  We first prove \eqref{eq:lem:no_bridge_GM}.  The
super-node $\mbox{\sc v}$ has no out-edges in $\Mgra{G_p}$ if and only if i)
each node in $I_{\mbox{\sc v}}$ has no bridge to a node in $V \setminus
I_{\mbox{\sc v}}$ and ii) $w_1$ and $w_2$ share no-ring edges with nodes in $V
\setminus I_{\mbox{\sc v}}$ in $G_p$. Condition i) above is equivalent to the
following: for every $x = 1,\ldots , \ell$, the node $u\in I_v$ at distance
$d(w_1, u) = x$ from $w_1$ has no bridge of length exceeding $x$ on the
clockwise side and of length exceeding $\ell - x$ on the counter-clockwise
side.

Since, from Lemma~\ref{lem:length_bridges}, the probability that a node has a
bridge of length larger than $x$ on one side is at most $1/(2x^{\alpha-1})$, 

\begin{align*}
    &\Prc{\deg_{\Mgra{G_p}}({\mbox{\sc v}}) = 0}\geq \left[(1-p)\cdot \prod_{x=1}^{\ell}\left(1-\frac{1}{2x^{\alpha-1}}\right)\right]^2 \\ 
    &\geq (1-p)^2\cdot e^{-2\sum_{x=1}^{\ell} \frac{1}{x^{\alpha-1}}}\geq (1-p)^2 e^{-2/(\alpha-2)}\,.
\end{align*}
We next prove \eqref{eq:lem:no_bridge_GM_longerthan1}. 
Let $v_i$ a node in $V \setminus I_{\mbox{\sc v}}$ at ring-distance $i+\ell$ from 
$I_{\mbox{\sc v}}$, i.e., such that $\min\{d(w_1, v_i), d(w_2, v_i)\} = i + \ell$. 
We have:
\begin{align*}
    &\Prc{\text{$v_i$ is not a neighbor of any node in $I_{\mbox{\sc v}}$ in 
    $G_p$}}\\ 
    &\geq \prod_{x=1}^{\ell}\left(1-\frac{1}{C(\alpha,n)(x+\ell+i)^\alpha}\right) 
    \geq e^{-\sum_{x=1}^{\ell}\frac{1}{(x+\ell+i)^{\alpha}}}
    \geq e^{-\frac{1}{(i+\ell)^{\alpha-1}}}\,.
\end{align*}
Then, let $\Mgra{(E_p)_2}$ denote the set of super-bridges in $\Mgra{G_p}$, we
use the above inequality to bound the expected number of super-bridges that are
incident in $\mbox{\sc v}$:
\begin{align*}
    &\Expcc{|\N_{\Mgra{G_p}}(\mbox{\sc v})\cap\Mgra{(E_p)_2}|} =2\sum_{i=1}^{n/2 - 
    \ell}\Prc{\text{$v_i$ has a neighbor in $I_{\mbox{\sc v}}$} }\\
    & \leq 2\sum_{i=1}^{+\infty} 
    1-e^{-\frac{1}{(i+\ell)^{\alpha-1}}}\leq 2\sum_{i=1}^{+\infty} 
    \frac{2}{(i+\ell)^{\alpha-1}}\leq\frac{2}{(\alpha-2)\ell^{\alpha-2}}.
\end{align*}
Finally, the proof follows from
\begin{align*}
    &\Prc{\text{\mbox{\sc v} has a super-bridge in $\Mgra{G_p}$}}\leq \Expcc{|\N_{\Mgra{G_p}}(\mbox{\sc v})\cap\Mgra{(E_p)_2}|}.
\end{align*}

\subsection{Proof of Lemma~\ref{lem:1alpha2_belowthreshold}}
\label{ssec:lem:1alpha2_belowthreshold}
Let $\underline{p}=1/3$ and consider an arbitrary node $s \in V$. We consider
an execution of the BFS in Algorithm \ref{alg:sequential-BFS-visit} with input
the percolation subgraph $G_p=(V,E_p)$ of $G=(V,E)$ and the source $s$. 

\begin{algorithm}[]
\caption{BFS visit of $G_p$}
\small{
\begin{algorithmic}[1]
\State\textbf{Input}: the subgraph  $G_p=(V,E_p)$, a  source $s \in V$
\State $Q = \{s\}$
\State $R = \emptyset$
\While{$Q \neq \emptyset $}\label{line:whileBFS}
    \State $w = \dequeue(Q)$
    \State $R= R \cup \{w\}$
			\For {each neighbor $x$ of $w$ in $G_p$ such that $x \not \in R $} 
					   \State $\enqueue(x,Q)$\label{line:addqueue}
				  \EndFor
 \EndWhile
\end{algorithmic}}
\label{alg:sequential-BFS-visit}
\end{algorithm}
We consider the generic $t$-th iteration of the while loop at line
\ref{line:whileBFS} and we denote by $W_t$ the number of nodes added to the
queue $Q$ at line \ref{line:addqueue}. Moreover,  $B_t$ is the set of nodes
that are in $R$ in the $t$-th iteration. By its definition, $B_t$ is a
branching process described by the following recursion: 
\begin{align}
    \begin{cases}
    B_t = B_{t-1} + W_t -1 &\quad \text{if $B_{t-1} \neq 0$}
    \\ B_t = 0  &\quad \text{if $B_{t-1}=0$}
    \\
    B_0 = 1.
    \end{cases}
\end{align}
Note that, from Definition~\ref{def:small-world}, each node $w \in V$ has
expected degree $\Expcc{\deg(w)}=3$. So, since each edge in $G$ is also in
$G_p$ with probability $p$, we have $\Expcc{W_1} = 3p$ and $\Expcc{W_t}\le 3p$
for $t > 1$.\footnote{We have not strict equality for $t > 1$, which follows
since the $W_t$'s are not independent in general, since
line~\ref{line:addqueue} is only executed if $x\not\in R$.} Since
$p<\underline{p}$, there is a constant $\delta$ such that $p=(1-\delta)/3$ and,
for each $t \geq 0$ 
\[
\Expcc{W_t} = 1-\delta.
\]
We consider the $T$-th iteration of the while loop, where $T=\gamma \log n$.
Note that the random variables $W_1,\ldots , W_T$ are not independent as noted
earlier but, as remarked in the proof of
Lemma~\ref{lem:termination_alg_alfa>2}, it is easy to show that they are
stochastically dominated by $T$ independent random variables distributed as
$W_1$. For the sake of simplicity, we abuse notation, by using $W_1,\dots, W_T$
to denote the  $T$ independent copies of $W_1$ in the remainder of this proof.
Now, if
\begin{equation*}
    \sum_{i=1}^{T}W_i -T<0,
\end{equation*}
then $B_T = 0$. We notice that each $W_i$ can be written as a sum of $n+2$
independent Bernoulli random variables.\footnote{Assume node $v$ is visited in
the $t$-iteration. Then we have $2$ indicator variables for the $2$ ring edges
incident in $v$, plus $n$ indicator variables, corresponding to $n$ bridges
potentially incident in $v$.} Hence, applying Chernoff's bound to
$W=\sum_{i=1}^T W_i$ we obtain:
\begin{equation*}
    \Prc{W >(1+\delta)\Expcc{W}}\leq e^{-\frac{\delta^2}{2}\Expcc{W}}.
\end{equation*}
Next, since $\Expcc{W}<(1-\delta)T$,
\begin{equation*}
    \Prc{W>(1-\delta^2)T}\leq e^{-\frac{\delta^2(1-\delta)}{2}T} \leq 
    \frac{1}{n^2},
\end{equation*}
where the last inequality follows if we take $\gamma \geq
4/(\delta^2(1-\delta))$. This allows us to conclude that 
\[
\Prc{B_T=0}\geq 1-\frac{1}{n^2},
\]
for some $T=\bigO(\log n)$, which in turn implies that with the above
probability, the connected component of which $v$ is part contains at most
$\bigO(\log n)$ nodes. Finally, a union bound over all nodes in $V$ concludes
the proof.

\subsection{Proof of Lemma~\ref{le:1alpha_low}}\label{ssec:1alpha_low}
We first prove the following fact.

\begin{fact}
\label{fact:alpha<1_erdos}
Under the hypotheses of Lemma~\ref{le:1alpha_low}, a constant $c \in (0,1)$
(depending only on $\alpha$) exists such that, for any  $u,v \in V$,   
\begin{equation*}
    \Prc{(u,v) \in E_p}\geq \frac{pc}{n}.
\end{equation*}
\end{fact}

\begin{proof}
The normalizing constant $C(\alpha,n)$ (see Definition~\ref{def:small-world})
can be upper bounded as follows
\[
C(\alpha,n) = 2\sum_{x=2}^{n/2}\frac{1}{x^\alpha} 
\leq \frac{2}{2^{\alpha}}+\int_{2}^{n/2}\frac{2}{x^{\alpha}}dx
\leq 2^{1-\alpha} +\frac{2}{1-\alpha}\left(\frac{n}{2}\right)^{1-\alpha}.
\]
Hence, a constant $c \in (0,1)$ (depending only on $\alpha$) exists such that

\begin{align*}
\Prc{(u,v) \in E_p} & = \frac{p}{C(\alpha,n)d(u,v)^\alpha} 
\geq \frac{p}{C(\alpha,n)(n/2)^\alpha} \\
& \geq \frac{p}{2^{1-\alpha}(n/2)^{\alpha}+n/(1-\alpha)}
\geq \frac{p \cdot c}{n}.
\end{align*}
\end{proof}

The above fact proves that, when $\alpha<1$, the percolation subgraph $G_p$ of
a graph sampled from \SWGpl stochastically dominates a cycle with additional
\Erdos\ random edges. To prove Lemma~\ref{le:1alpha_low} we thus use a previous
result in~\cite{becchettiCDPTZ21} on such class of random graphs. In
particular, we first give an equivalent formulation of Lemma~5.1
in~\cite{becchettiCDPTZ21}, that states that, with constant probability the
sequential BFS visit (Algorithm~\ref{alg:sequential-BFS-visit_new}) reaches
$\Omega(\log n)$ nodes within $\bigO(\log n)$ rounds. We then consider a
\textit{parallel}-BFS visit (Algorithm~\ref{alg:parallel-BFS-visit}) and give
an equivalent formulation of Lemma~5.2 in~\cite{becchettiCDPTZ21}, that states
that, w.h.p., the parallel-BFS visit starting with $\Omega(\log n)$ nodes
reaches a constant fraction of nodes within $\bigO(\log n)$ rounds.

We also introduce a slight different version of the BFS visit of
Algorithm~\ref{alg:parallel-BFS-visit}, where we have also a set $R_0$ of
removed nodes in input. 

\begin{algorithm}[H]
\caption{BFS visit of $G_p$}
\small{
\begin{algorithmic}[1]
\State\textbf{Input}: the subgraph $G_p=(V,E_p)$, an initiator $s \subseteq V$, a set of removed nodes $R_0 \subseteq V$.
\State $Q = \{s\}$
\State $R = R_0$
\While{$Q \neq \emptyset $}
    \State $w = \dequeue(Q)$
    \State $R= R \cup \{w\}$
			\For {each neighbor $x$ of $w$ in $G_p$ such that $x \not \in R $} 
					   \State $\enqueue(x,Q)$
				  \EndFor
 \EndWhile
\end{algorithmic}}
\label{alg:sequential-BFS-visit_new}
\end{algorithm}

\begin{lemma}
\label{lem:old_erdos_phase1}
Under the hypothesis of Lemma~\ref{le:1alpha_low}, let $v \in V$ be a node and
$c \in (0,1)$ be a constant as in Fact~\ref{fact:alpha<1_erdos}. For every
$\beta>0$, $\varepsilon>0$, and percolation probability
$p>\frac{\sqrt{c^2+6c+1}-c-1}{2c}+\varepsilon$, there are positive parameters
$k$ and $\gamma$ (depending only on $\varepsilon$, $c$ and $p$) such that the
following holds: the BFS visit (Algorithm~\ref{alg:sequential-BFS-visit_new})
with input $G_p$, $v$, and a set $R_0$ with $|R_0| \leq \log^4 n$, with
probability $\gamma$, a time $\tau_1=\bigO(\log n)$ exists such that
\[
|(R\setminus R_0)\cup Q|\geq n/k \qquad \text{ OR } \qquad |Q|\geq \beta \log n \, .
\] 
\end{lemma}

\begin{algorithm}[H]
\caption{Parallel BFS visit of $G_p$}
\small{
\textbf{Input}: the subgraph $G_p =(V,E_p)$, a set of initiators $I_0 \subseteq V$, a set of removed nodes $R_0 \subseteq V$

    \begin{algorithmic}[1]
	\State $Q = I_0$
	\State $R= R_0$
	\While{$Q \neq \emptyset $}\label{line:whileBFS-parallel}
	   \State $A=R\cup Q$
	   \State $X=\text{neighbors}(Q)$
	   \State $Q'=Q$
	   \State $Q=\emptyset$
	   \While{$Q' \neq \emptyset$}
	        \State $w=\texttt{dequeue}(Q')$
	        \State $R= R \cup \{w\}$
	       \For{each $x \in X$}
	       \State $\texttt{enqueue}(x,Q)$
	       \EndFor
	        \EndWhile
 	 \EndWhile
	\end{algorithmic}}
\label{alg:parallel-BFS-visit}
\end{algorithm}

\begin{lemma}
\label{lem:old_erdos_phase2}
Under the hypothesis of Lemma~\ref{le:1alpha_low}, let $c \in (0,1)$ be a
constant as in Fact~\ref{fact:alpha<1_erdos}. For every $\varepsilon>0$ and
percolation probability $p>\frac{\sqrt{c^2+6c+1}-c-1}{2c}+\varepsilon$, there
are positive parameters $k$, $\beta$ (depending only on $c$, $p$ and
$\varepsilon$) such that the following holds: For any pair of sets $I_0, R_0
\subseteq V$, with $|I_0| \geq \beta \log n$ and $|R_0| \leq \log^4 n$, in the
parallel BFS-visit (Algorithm~\ref{alg:parallel-BFS-visit}) with input $G_p$,
$I_0$, and $R_0$, with probability at least $1-1/n$, a time $\tau_2=\bigO(\log
n)$ exists such that 
\[
|(R\setminus R_0)\cup Q|\geq n/k . 
\]
\end{lemma}

Now we are ready to prove Lemma~\ref{le:1alpha_low}.

\begin{proof}[Proof of Lemma~\ref{le:1alpha_low}]
Let $c$ be as in Fact~\ref{fact:alpha<1_erdos} and $\overline{p} =
\frac{\sqrt{c^2+6c+1}-c-1}{2c}$. Let $p > \overline{p} + \varepsilon$, for an
arbitrarily-small constant $\varepsilon > 0$, and let $\beta>0$ be the constant
in Lemma~\ref{lem:old_erdos_phase2} and $k$ be the constant in
Lemmas~\ref{lem:old_erdos_phase1} and~\ref{lem:old_erdos_phase2}.  We consider
the following process, where we initialize $R_0 = \emptyset$ and $\tau_1 =
\bigO(\log n)$ is as in Lemma~\ref{lem:old_erdos_phase1}.

\begin{enumerate}
    \item Consider a node $v \in V \setminus R_0$.

    \item From $v$, perform a sequential-BFS visit
    (Algorithm~\ref{alg:sequential-BFS-visit}), with input $G_p$, $v$, and
    $R_0$, for $\tau_1$ while loops and add to $R_0$ the sets $Q$ and $R$ as
    they are at the end of the $\tau_1$-th iteration of the while loop ($R_0 =
    R_0 \cup R \cup Q$).
    
    \item If $|Q| \geq \beta \log n$ or $|Q \cup R| \geq n/k$, interrupt the
    process.
    
    \item Restart from $1$.
\end{enumerate}

Let $\gamma>0$ be the constant in Lemma~\ref{lem:old_erdos_phase1}.  We prove
that the process above terminates within $\sigma = \log_{1-\gamma}(n)$
iterations, w.h.p. 

First we notice that, at each iteration of the process, the set $R_0$ grows,
w.h.p., at most of size $\bigO(\log^2 n)$, since each node in $G_p$ has degree
at most $\bigO(\log n)$, w.h.p. (Fact~\ref{fact:degreeofSWGpl}) and so, in
$\tau_1$ iteration of the parallel-BFS, at most $\bigO(\log^2 n)$ nodes will be
reached by $v$. This implies that, at each iteration $i \leq \sigma$ of the
process $|R_0| = \bigO(\log^3 n)$, w.h.p.

From Lemma~\ref{lem:old_erdos_phase1} it follows that,a sequential-BFS with in
input $G_p$, $v$ and $R_0$ with $|R_0| =\bigO(\log^3 n)$ is such that, at the
end of $\tau_1$-th iteration
\[
\Prc{|Q|\geq \beta \log n \text{ or } |R \cup Q| \geq n/k}\geq \gamma>0\,.
\]

Therefore, the probability that the process exceeds $\sigma$ iterations
is at most $(1-\gamma)^\sigma \leq 1/n$.

So, w.h.p., a node $v$ exists such that the sequential-BFS starting from $v$,
after $\bigO(\log n)$ steps, satisfies at least one of the two conditions: i)
$|Q| \geq \beta \log n$ or ii) $|Q \cup R | \geq n/k$.

If ii) holds, the lemma is proven. Indeed, w.h.p. we have the existence of a
node $v$ such that there is a set of $\Omega(n)$ nodes at distance at most
$\bigO(\log n)$ from $v$.

If i) holds, it suffices to perform a sequential-BFS
(Algorithm~\ref{alg:sequential-BFS-visit}) with in input $G_p$, $I_0 = Q$ and
$R_0$ and apply Lemma~\ref{lem:old_erdos_phase2} to claim that such BFS reaches
at least $\Omega(n)$ nodes in $\bigO(\log n)$ steps.  \end{proof}

\section{Mathematical Tools}
\begin{theorem}[Chernoff Bound, \cite{dubhashipanconesi09}]
\label{thm:chernoff}
Let $X=\sum_{i=1}^n X_i$, where $X_i$ with $i \in [n]$ are independently
distributed in $[0,1]$. Let $\mu=\Expcc{X}$ and $\mu_L \leq \mu \leq \mu_H$.
Then:
\begin{itemize}
    \item for  every  $t>0$
    \[
    \Prc{X>\mu_H+t},\Prc{X<\mu_L-t}\leq e^{-2t^2/n} ;
    \]
    \item for $\varepsilon>0$
    \[
    \Prc{X>(1+\varepsilon)\mu_H}\leq e^{-\frac{\epsilon^2}{2}\mu_H} 
    \ \text{ and } \
    \Prc{X<(1-\epsilon)\mu_L}\leq e^{-\frac{\epsilon^2}{2}\mu_L} . 
    \]
\end{itemize}

\end{theorem}

\end{document}